\documentclass[english]{article}
\usepackage[margin=2.5cm]{geometry}
\usepackage{hyperref}
\usepackage{amsthm,amsfonts,amssymb,amsmath}
\usepackage{graphicx}

\newtheorem{theorem}{Theorem}[section]
\newtheorem{proposition}[theorem]{Proposition}
\newtheorem{lemma}[theorem]{Lemma}

\newtheorem{claim}[theorem]{Claim}
\newtheorem{corollary}[theorem]{Corollary}

\theoremstyle{definition}
\newtheorem{definition}[theorem]{Definition}
\newtheorem{remark}[theorem]{Remark}

\numberwithin{equation}{section}
\numberwithin{figure}{section}
\allowdisplaybreaks

\usepackage{babel}
\usepackage[T1]{fontenc}
\usepackage[utf8]{inputenc}

\newcommand{\su}{\subseteq}
\newcommand{\sm}{\setminus}
\renewcommand{\l}{\ell}
\newcommand{\F}{\mathbb{F}}
\newcommand{\Fp}{\mathbb{F}_p}
\newcommand{\Fpn}{\mathbb{F}_p^{n}}
\newcommand{\Z}{\mathbb{Z}}

\newcommand{\spn}{\operatorname{span}}
\newcommand{\proj}{\operatorname{proj}}
\newcommand{\PP}{\mathbb{P}}
\newcommand{\EE}{\mathbb{E}}

\begin{document}
\title{Finding solutions with distinct variables to systems of linear equations over $\mathbb{F}_p$}
\author{Lisa Sauermann\thanks{School of Mathematics, Institute for Advanced Study, Princeton, NJ 08540. Email: {\tt lsauerma@mit.edu}. Research supported by NSF Grant CCF-1900460 and NSF Award DMS-2100157.}}

\maketitle

\begin{abstract}
Let us fix a prime $p$ and a homogeneous system of $m$ linear equations $a_{j,1}x_1+\dots+a_{j,k}x_k=0$ for $j=1,\dots,m$ with coefficients $a_{j,i}\in\mathbb{F}_p$. Suppose that $k\geq 3m$, that $a_{j,1}+\dots+a_{j,k}=0$ for $j=1,\dots,m$ and that every $m\times m$ minor of the $m\times k$ matrix $(a_{j,i})_{j,i}$ is non-singular. Then we prove that for any (large) $n$, any subset $A\subseteq\mathbb{F}_p^n$ of size $|A|> C\cdot \Gamma^n$ contains a solution $(x_1,\dots,x_k)\in A^k$ to the given system of equations such that the vectors $x_1,\dots,x_k\in A$ are all distinct. Here, $C$ and $\Gamma$ are constants only depending on $p$, $m$ and $k$ such that $\Gamma<p$. 

The crucial point here is the condition for the vectors $x_1,\dots,x_k$ in the solution $(x_1,\dots,x_k)\in A^k$ to be distinct. If we relax this condition and only demand that $x_1,\dots,x_k$ are not all equal, then the statement would follow easily from Tao’s slice rank polynomial method. However, handling the distinctness condition is much harder, and requires a new approach. While all previous combinatorial applications of the slice rank polynomial method have relied on the slice rank of diagonal tensors, we use a slice rank argument for a non-diagonal tensor in combination with combinatorial and probabilistic arguments.
\end{abstract}

\section{Introduction}

Given a linear system of equations with coefficients in $\Fp$ for some fixed prime $p$, what is the largest size of a subset $A\su \Fpn$ which does not contain a (non-trivial) solution to the given linear system of equations? This is a fundamental question in additive combinatorics, and it can be viewed as the finite field analog of similar questions for subsets $A\su \{1,\dots,N\}$ whose history dates back many decades (see e.g.\ \cite{erdos-turan-1936,komlos-sulyok-szemeredi-1975}). For example, a $k$-term arithmetic progression can be described by a system of $k-2$ linear equations (with $k$ variables), and bounding the largest possible size of a $k$-term-progression-free subset in $\Fpn$ is an intensively studied and still wide open problem \cite{bateman-katz, edel, ellenberg-gijswijt, elsholtz-pach,  green-tao-1, green-tao-2, green-tao-2a, lin-wolf, meshulam}.

Let us from now on fix a prime $p$ and consider a linear system of $m$ equations in $k$ variables of the form
\begin{align}
a_{1,1}x_1+\dots+a_{1,k}x_k&=0\notag\\
\vdots\qquad\qquad&\tag{$\star$}\\
a_{m,1}x_1+\dots+a_{m,k}x_k&=0\notag
\end{align}
with coefficients $a_{j,i}\in \Fp$ for $j=1,\dots,m$ and $i=1,\dots,k$. For large $n$, we are interested in the largest possible size of a subset $A\su \Fpn$ such that there is no (non-trivial) solution $(x_1,\dots,x_k)\in A^k$ to ($\star$).

If we have $a_{j,1}+\dots+a_{j,k}\neq 0$ for some $j\in \{1,\dots,m\}$ (i.e.\ if for one of the $m$ equations the coefficients do not sum up to zero), then it is easy to see that there exists a subset $A\su \Fpn$ of size $p^{n-1}=(1/p)\cdot p^n$ such that the system ($\star$) does not have any solutions $(x_1,\dots,x_k)\in A^k$ (indeed, we can take $A$ to be the set of all vectors in $\Fpn$ whose first coordinate is $1$). For fixed $p$, this means that up to constant factors $A$ can be as large as the entire space $\Fpn$. However, the problem becomes much more interesting when $a_{j,1}+\dots+a_{j,k}=0$ for $j=1,\dots,m$.

So let us from now on assume that $a_{j,1}+\dots+a_{j,k}=0$ for $j=1,\dots,m$. Note that then $(x\dots,x)$ is a solution to the system ($\star$) for every $x\in A$. However, we can ask for the largest size of a subset $A\su \Fpn$ without a solution $(x_1,\dots,x_k)\in A^k$ to ($\star$) where $x_1,\dots,x_k$ are not all equal. It is a highly non-trivial result to show that we must have $|A|=o(p^n)$ as $n\to \infty$ (this can for example be deduced from \cite[Theorem 1]{kral-serra-vena} or \cite[Theorem 2.2]{shapira}, both of which rely on deep results on hypergraph regularity, and the actual quantitative bounds obtained this way are very poor).

Building upon a breakthrough of Ellenberg and Gijswijt \cite{ellenberg-gijswijt} on bounding the size of $3$-term-progression-free subsets of $\Fpn$ as well as on prior work of Croot, Lev, and Pach \cite{croot-lev-pach}, Tao \cite{tao} introduced a new polynomial method, which is now called the \emph{slice rank polynomial method}. This method immediately gives much stronger upper bounds on the size of $A$ in the question above if we assume that the number $k$ of variables in the system ($\star$) is sufficiently large with respect to the number $m$ of equations. In fact, assuming that $k\geq 2m+1$, one can prove that the size of $A$ needs to be \emph{exponentially smaller} than $p^n$ if there is no solution $(x_1,\dots,x_k)\in A^k$ to ($\star$) where $x_1,\dots,x_k$ are not all equal.

\begin{theorem}[Tao]\label{theo-tao}
For any fixed integers $m\geq 1$ and $k\geq 2m+1$ and a fixed prime $p$, there exists a constant $1\leq \Gamma_{p,m,k}<p$ such that the following holds: For any coefficients $a_{j,i}\in \Fp$ for $j=1,\dots,m$ and $i=1,\dots, k$ with $a_{j,1}+\dots+a_{j,k}=0$ for $j=1,\dots,m$, for any non-negative integer $n$ and any subset $A\su \Fpn$ of size $|A|> (\Gamma_{p,m,k})^n$, the system ($\star$) has a solution $(x_1,\dots,x_k)\in A^k$ such that the vectors $x_1,\dots,x_k\in A$ are not all equal.
\end{theorem}

As mentioned above, Theorem \ref{theo-tao} follows immediately from Tao's slice rank polynomial method  \cite{tao}. For the reader's convenience we will present the proof in Section \ref{sect-background}.

Recall that our original question was to bound the size of a subset $A\su \Fpn$ which does not contain a non-trivial solution to the system ($\star$). Obviously, it needs to be specified what we mean by ``non-trivial'' here. If a solution  $(x_1,\dots,x_k)\in A^k$ is considered non-trivial as soon as $x_1,\dots,x_k\in A$ are not all equal, then Theorem \ref{theo-tao}  provides a strong bound for $|A|$ as long as $k\geq 2m+1$. However, there are also several other notions of ``non-trivial'' solutions to a linear equation or a system of linear equations (see for example Ruzsa's work \cite{ruzsa-1, ruzsa-2}), and a particularly natural such notion is to demand for a ``non-trivial'' solution $(x_1,\dots,x_k)\in A^k$ to consist of \emph{distinct} $x_1,\dots,x_k$.

In the integer setting, the problem of bounding the largest size of a subset $A\su \{1,\dots,N\}$ with no solution $(x_1,\dots,x_k)\in A^k$ to a given linear system of equations where $x_1,\dots,x_k$ are distinct has already been considered almost fifty years ago \cite{komlos-sulyok-szemeredi-1975}. Here, we consider the same problem in the setting of $\Fpn$, i.e.\ we are asking for the largest size of a subset $A\su \Fpn$ which does not contain a solution $(x_1,\dots,x_k)\in A^k$ to ($\star$) with distinct $x_1,\dots,x_k$.

Similar to Theorem \ref{theo-tao} above, our main result states that if the number $k$ of variables in the system ($\star$) is sufficiently large with respect to the number $m$ of equations and if the system ($\star$) is reasonably generic, then the size of $A$ must be exponentially smaller than $p^n$.

\begin{theorem}\label{theo-distinct}
For any fixed integers $m\geq 1$ and $k\geq 3m$ and a fixed prime $p$, there exist constants $C_{p,m,k}\geq 1$ and $1\leq \Gamma_{p,m,k}^*<p$ such that the following holds: Let $a_{j,i}\in \Fp$ for $j=1,\dots,m$ and $i=1,\dots, k$ be coefficients with $a_{j,1}+\dots+a_{j,k}=0$ for $j=1,\dots,m$ such that every $m\times m$ minor of the $m\times k$ matrix $(a_{j,i})_{j,i}$ is non-singular. Then for any non-negative integer $n$ and any subset $A\su \Fpn$ of size $|A|> C_{p,m,k}\cdot (\Gamma_{p,m,k}^*)^n$, the system  ($\star$) has a solution $(x_1,\dots,x_k)\in A^k$ such that the vectors $x_1,\dots,x_k\in A$ are all distinct.
\end{theorem}

While it may seem at first that it should not make much of a difference whether one demands $x_1,\dots,x_k\in A$ to be distinct or to be not all equal, it is in fact much more difficult to prove the existence of a solution $(x_1,\dots,x_k)\in A^k$ with distinct $x_1,\dots,x_k$ as in Theorem \ref{theo-distinct}. In fact, in some cases demanding such a distinct solution can require qualitatively different bounds for the size of $A$ (see Theorems 2 and 4 in \cite{mimura-tokushige-1}, and see also Theorems 3.2 and 3.3 in \cite{ruzsa-1}). The slice rank polynomial method argument proving Theorem \ref{theo-tao} completely breaks down when requiring $x_1,\dots,x_k$ to be distinct, and new ideas are required in order to prove Theorem \ref{theo-distinct}. In particular, our proof uses a variant of the slice rank polynomial method that has not appeared in combinatorial applications before.

The problem of bounding the size of a subset of $A\su \Fpn$ not containing a solution with distinct variables to a given linear equation or system of linear equations has been studied by various authors in the past. In the case of the single equation $x_1+\dots+x_p=0$, which is very closely related to the Erd\H{o}s-Ginzburg-Ziv problem in discrete geometry, Naslund \cite{naslund} proved a bound of the form $|A|\leq C_p\cdot (\Gamma_p)^n$ for some $\Gamma_p$ between $0.84p$ and $0.92p$ (the constant $C_p$ was later improved in \cite{fox-sauermann}). The best current bound in this case is $|A|\leq C_p\cdot (2\sqrt{p})^n$ due to the author \cite{sauermann}.

In the case of a general single linear equation $a_1x_1+\dots+a_kx_k=0$ with $a_1+\dots+a_k=0$ (i.e.\ in the case of $m=1$), Theorem \ref{theo-distinct} has been proved by Mimura and Tokushige \cite[Theorem 1]{mimura-tokushige-2}, and this case also follows from \cite[Theorem 5.1]{sauermann} with significantly better bounds for $\Gamma_{p,1,k}^*$. Different specfic examples of linear systems of multiple equations have been considered in \cite{mimura-tokushige-jcta, mimura-tokushige-1, mimura-tokushige-2}, but Theorem \ref{theo-distinct} is the first relatively general result for systems of multiple equations.

It is worth noting that one cannot hope to remove the assumption in Theorem \ref{theo-distinct} on the non-singularity of the $m\times m$ minors of the matrix $(a_{j,i})_{j,i}$. Indeed, suppose that the equation $x_1-x_2=0$ was part of the system ($\star$) or could be obtained as a linear combination of the equations in ($\star$). Then clearly ($\star$) does not have any solutions $(x_1,\dots,x_k)$ with distinct $x_1,\dots,x_k$. It may potentially be possible to weaken the non-singularity assumption in a way that excludes such situations (see also the discussion in Section \ref{sect-concluding-remarks}), but this example shows that some assumption is necessary.

The proof of Theorem \ref{theo-distinct} relies on probabilistic subspace sampling arguments and combinatorial ideas, as well as on a new way to apply the slice rank polynomial method. A key step for proving Theorem \ref{theo-distinct} will be to show the following theorem.

\begin{theorem}\label{theo-rank}
For any fixed integers $m\geq 1$, $r\geq 2$ and $k\geq 2m+r-1$ and a fixed prime $p$, there exist constants $C_{p,m,k,r}^{\textnormal{rank}}\geq 1$ and $1\leq \Gamma_{p,m,k,r}^{\textnormal{rank}}<p$ such that the following holds: For any coefficients $a_{j,i}\in \Fp$ for $j=1,\dots,m$ and $i=1,\dots, k$ with $a_{j,1}+\dots+a_{j,k}=0$ for $j=1,\dots,m$, for any non-negative integer $n$ and any subset $A\su \Fpn$ of size $|A|> C_{p,m,k,r}^{\textnormal{rank}}\cdot (\Gamma_{p,m,k,r}^{\textnormal{rank}})^n$, the system  ($\star$) has a solution $(x_1,\dots,x_k)\in A^k$ such that  the subspace $\spn(x_1,\dots,x_k)\su \Fpn$ spanned by the vectors $x_1,\dots,x_k$ has dimension $\dim\spn(x_1,\dots,x_k)\geq r$.
\end{theorem}

Theorem \ref{theo-rank} can be viewed as a ``high-rank'' generalization of Tao's slice rank method result in Theorem \ref{theo-tao} (where one wants to find a solution $(x_1,\dots,x_k)\in A^k$ of high rank). Indeed, taking $r=2$ in Theorem \ref{theo-rank} implies Theorem \ref{theo-tao}. Also note that unlike in Theorem \ref{theo-distinct}, in Theorem \ref{theo-rank} we do not require any genericity assumption on the matrix $(a_{j,i})_{j,i}$.

\textit{Organization.} In the next section we will deduce Theorem \ref{theo-distinct} from Theorem \ref{theo-rank}. Section \ref{sect-background} contains some background on the slice rank polynomial method (including a proof of Theorem \ref{theo-tao}), and discusses the new way in which it will be used in the proof of Theorem \ref{theo-rank}. The proof of Theorem \ref{theo-rank} is in Section \ref{sect-proof-rank}. We finish with some concluding remarks in Section \ref{sect-concluding-remarks}.

\textit{Notation.} As usual, we write $[k]=\{1,\dots,k\}$.  For a subset $I\su [k]$, an \emph{$I$-tuple} of vectors in $\Fpn$ is a tuple $(x_i\mid i\in I)$ indexed by the set $I$  (with $x_i\in \Fpn$ for all $i\in I$). For example, a $[k]$-tuple is simply a $k$-tuple $(x_1,\dots,x_k)$ and a $\{1,2,4\}$-tuple is a tuple $(x_1, x_2, x_4)$.

For a vector space $V$ and a subspace $U$, we can consider the quotient space $V/U$. We denote the projection of a vector $x\in V$ onto this quotient space by $\proj_{V/U}(x)$. Note that $\proj_{V/U}(x)$ is a vector in $V/U$ and it is non-zero if and only if $x\not\in U$.

\section{Proof of Theorem \ref{theo-distinct} assuming Theorem \ref{theo-rank}}
\label{sect-proof-distinct}

We deduce Theorem \ref{theo-distinct} from Theorem \ref{theo-rank} using an inductive argument. More precisely, we will show the following theorem by induction on $\l$ (where the base case $\l=m+1$ will be obtained from Theorem \ref{theo-rank}, and the final case $\l=k$ will give the desired statement in Theorem \ref{theo-distinct}).

\begin{theorem}\label{theo-distinct-induction}
For any fixed integers $m\geq 1$ and $k\geq 3m$ and $m+1\leq \l\leq k$ as well as a fixed prime $p$, there exist constants $C\geq 1$ and $0<c\leq 1$ such that the following holds: Let $a_{j,i}\in \Fp$ for $j=1,\dots,m$ and $i=1,\dots, k$ be coefficients with $a_{j,1}+\dots+a_{j,k}=0$ for $j=1,\dots,m$ such that every $m\times m$ minor of the $m\times k$ matrix $(a_{j,i})_{j,i}$ is non-singular. Then for any non-negative integer $n$ and any subset $A\su \Fpn$ of size $|A|> C\cdot p^{(1-c)n}$, the system  ($\star$) has a solution $(x_1,\dots,x_k)\in A^k$ such that $\dim\spn(x_1,\dots,x_k)\geq m+1$ and such that among  $x_1,\dots,x_k\in A$ there are at least $\l$ distinct vectors.
\end{theorem}

Note that Theorem \ref{theo-distinct} follows from Theorem \ref{theo-distinct-induction} for $\l=k$. Indeed, taking $C_{p,m,k}=C$ and $\Gamma_{p,m,k}^*=p^{1-c}<p$ (for the constants $C$ and $c$ obtained in Theorem \ref{theo-distinct-induction} for $\l=k$), we can see that every subset $A\su \Fpn$ of size $|A|> C_{p,m,k}\cdot (\Gamma_{p,m,k}^*)^n=C\cdot p^{(1-c)n}$ contains a solution $(x_1,\dots,x_k)\in A^k$ to the system $(\star)$ such that among  $x_1,\dots,x_k\in A$ there are at least $k$ distinct vectors (meaning that the vectors $x_1,\dots,x_k\in A$ are all distinct).

In the proof of Theorem \ref{theo-distinct-induction} we will use a probabilistic subspace sampling argument. This will require the following easy lemma.

\begin{lemma}\label{lemma-subspace-probability}
Let $n>0$ and $0\leq d\leq n$ be integers, and let $y_1,\dots,y_s\in \Fpn$ be linearly independent vectors in $\Fpn$. Then for a uniformly random $d$-dimensional subspace $V\su \Fpn$, we have
\[\PP[y_1,\dots,y_s\in V]=\frac{p^d-1}{p^n-1}\cdot \frac{p^d-p}{p^n-p}\dotsm \frac{p^d-p^{s-1}}{p^n-p^{s-1}}\leq \left(\frac{p^d}{p^n}\right)^s\]
\end{lemma}
\begin{proof}
Let us first check  the inequality between the term in the middle and the term on the right-hand side. If $s>d$, then the middle term is zero and therefore the inequality is true. If $s\leq d$, then the $s$ factors of the middle term are all positive and each of them is at most $(p^d-1)/(p^n-1)\leq p^d/p^n$, which also implies the inequality.

Let us now prove by induction on $s$ that the probability on the left-hand side equals the term in the middle. For the base case $s=1$, note that $y_1\neq 0$ and that each of the $p^n-1$ non-zero vectors in $\Fpn$ is equally likely to be contained in $V$. Since $V$ always contains exactly $p^d-1$ non-zero vectors, we can conclude that $\PP[y_1\in V]=(p^d-1)/(p^n-1)$ as desired.

Now suppose that $s\geq 2$ and that we have already proved
\[\PP[y_1,\dots,y_{s-1}\in V]=\frac{p^d-1}{p^n-1}\cdot \frac{p^d-p}{p^n-p}\dotsm \frac{p^d-p^{s-2}}{p^n-p^{s-2}}.\]
Then it suffices to show that $\PP[y_s\in V\mid y_1,\dots,y_{s-1}\in V]=(p^d-p^{s-1})/(p^n-p^{s-1})$. Note that by the assumption on $y_1,\dots,y_s$ being linearly independent, the vector $y_s$ does not lie in the span of $y_1,\dots,y_{s-1}$. Furthermore, this span is $(s-1)$-dimensional and therefore consists of exactly $p^{s-1}$ vectors. Hence there are exactly $p^n-p^{s-1}$ vectors outside $\spn(y_1,\dots,y_s)$ and each of these vectors is equally likely to be contained in $V$ when conditioning on the event $y_1,\dots,y_{s-1}\in V$. Since under this conditioning, $V$ always contains exactly $p^d-p^{s-1}$ vectors outside $\spn(y_1,\dots,y_s)$, we can conclude that $\PP[y_s\in V\mid y_1,\dots,y_{s-1}\in V]=(p^d-p^{s-1})/(p^n-p^{s-1})$ as desired.\end{proof}

Let us now prove Theorem \ref{theo-distinct-induction} by induction on $\l$, assuming that Theorem \ref{theo-rank} is true.

\begin{proof}[Proof of Theorem \ref{theo-distinct-induction}]
The base case $\l=m+1$ follows easily from Theorem \ref{theo-rank} for $r=m+1$. Indeed, we have $k\geq 3m=2m+r-1$, so there are constants $C_{p,m,k,m+1}^{\textnormal{rank}}\geq 1$ and $1\leq \Gamma_{p,m,k,m+1}^{\textnormal{rank}}<p$ such that the statement in Theorem \ref{theo-rank} holds. Now, let $C=C_{p,m,k,m+1}^{\textnormal{rank}}$ and let $0<c\leq 1$ be such that $p^{1-c}=\Gamma_{p,m,k,m+1}^{\textnormal{rank}}$. Then for any subset any subset $A\su \Fpn$ of size $|A|> C\cdot p^{(1-c)n}=C_{p,m,k,m+1}^{\textnormal{rank}}\cdot (\Gamma_{p,m,k,m+1}^{\textnormal{rank}})^n$, the system  ($\star$) has a solution $(x_1,\dots,x_k)\in A^k$ such that
$\dim\spn(x_1,\dots,x_k)\geq m+1$. Note that the condition $\dim\spn(x_1,\dots,x_k)\geq m+1$ automatically implies that there must be  at least $\l=m+1$ distinct vectors among  $x_1,\dots,x_k$. This proves Theorem \ref{theo-distinct-induction} for $\l=m+1$.

Now let us assume that $m+2\leq \l\leq k$ and that Theorem \ref{theo-distinct-induction} holds for $\l-1$ with constants $C'\geq 1$ and $0<c'\leq 1$. We will prove that then Theorem \ref{theo-distinct-induction} also holds for $\l$. Let us take
\[c=\frac{1}{(m/c')+1}\]
and note that
\begin{equation}\label{eq-new-old-c-distinct}
1-(m+1)c=\frac{(m/c')+1-(m+1)}{(m/c')+1}=m\cdot \frac{(1/c')-1}{(m/c')+1}=(1-c')\cdot \frac{(m/c')}{(m/c')+1}=(1-c')(1-c).
\end{equation}

Our goal is to show that the statement in Theorem \ref{theo-distinct-induction} holds for some constant $C\geq 1$ only depending on $m$, $k$, $\ell$ and $p$. By making the constant $C$ sufficiently large we may assume that $n\geq m/(1-c)$. We may furthermore assume that $0\not\in A$ (otherwise we can delete the zero-vector from $A$ and account for that by increasing the constant $C$ by $1$).

Hence it suffices to prove that for any $n\geq  m/(1-c)$ and any subset $A\su \Fpn$ of size $|A|> 2p(C'+k^22^k)\cdot p^{(1-c)n}$ with $0\not\in A$, the system  ($\star$) has a solution $(x_1,\dots,x_k)\in A^k$ satisfying the conditions in Theorem \ref{theo-distinct-induction}.

So let us assume for contradiction that $n\geq  m/(1-c)$ and that $A\su \Fpn$ is a subset of size
\[|A|> 2p(C'+k^22^k)\cdot p^{(1-c)n}\]
with $0\not\in A$, in which we cannot find a solution $(x_1,\dots,x_k)\in A^k$ to the system  ($\star$) with $\dim\spn(x_1,\dots,x_k)\geq m+1$ and such that among  $x_1,\dots,x_k\in A$ there are at least $\l$ distinct vectors.

The following definition of an $I$-interesting $I$-tuple plays a crucial role for the proof. Roughly speaking, our strategy will be to sample a random  subspace $V\su \Fpn$ of a suitably chosen dimension and to find a fairly large subset $A^*\su A\cap V$ which does not contain any $I$-interesting $I$-tuple. On the other hand, given that $A^*\su V$ is a large subset of the vector space $V$, by the induction hypothesis for $\l-1$, the set $A^*$ needs to contain a solution $(x_1,\dots,x_k)\in (A^*)^k$ to the system $(\star)$ with $\dim\spn(x_1,\dots,x_k)\geq m+1$ and such that among  $x_1,\dots,x_k\in A$ there are at least $\l-1$ distinct vectors. We will see that having such a solution $(x_1,\dots,x_k)\in (A^*)^k$ will actually force the set $A^*$ to contain some $I$-interesting $I$-tuple. This contradiction will finish the proof of the induction step.

\begin{definition}\label{defi-interesting}
For a subset $I\su [k]$ of size $|I|=m+1$, let us say that an $I$-tuple $(x_i\mid i\in I)$ of vectors $x_i\in A$ for $i\in I$ is \emph{$I$-interesting} if the vectors $x_i$ for $i\in I$ are linearly independent and the $I$-tuple $(x_i\mid i\in I)$ can be extended to a solution $(x_1,\dots,x_k)\in A^k$ to the system ($\star$) such that among the $k-m-1$ vectors $x_j$ for $j\in [k]\sm I$ there are at least $\l-m-1$ distinct vectors.
\end{definition}

The following claim states, roughly speaking, that every solution $(x_1,\dots,x_k)\in A^k$ satisfying the conditions in the induction hypothesis for $\l-1$ must contain an $I$-interesting $I$-tuple.

\begin{claim}\label{claim-contains-interesting}
Let $(x_1,\dots,x_k)\in A^k$ be a solution to the system $(\star)$ with $\dim\spn(x_1,\dots,x_k)\geq m+1$ and such that among  $x_1,\dots,x_k\in A$ there are exactly $\l-1$ distinct vectors. Then there is a subset $I\su [k]$ of size $|I|=m+1$ such that the $I$-tuple $(x_i\mid i\in I)$ is $I$-interesting.
\end{claim}
\begin{proof}
Since $\l-1\leq k-1$, at least one vector must be repeated among $x_1,\dots,x_k$. So let $a,b\in [k]$ be distinct indices such that $x_a=x_b$. Note that $x_a\neq 0$ by our assumption that $0\not\in A$. Hence, since $\dim\spn(x_1,\dots,x_k)\geq m+1$, we can extend the single vector $x_a$ to a list of $m+1$ linearly independent vectors chosen among $x_1,\dots,x_k$. In other words, we can find a subset $I\su [k]$ of size $|I|=m+1$ with $a\in I$ such that the vectors $x_i$ for $i\in I$ are linearly independent (and therefore in particular distinct). It now suffices to check that among the $k-m-1$ vectors $x_j$ for $j\in [k]\sm I$ there are at least $\l-m-1$ distinct vectors. Since $x_a=x_b$ and the vectors $x_i$ for $i\in I$ are distinct, we have $b\in [k]\sm I$. Recall that among $x_1,\dots,x_k\in A$ there are exactly $\l-1$ distinct vectors. When omitting the $m+1$ vectors $x_i$ with $i\in I$, at most $m$ of these $\l-1$ distinct vectors disappear (since the vector $x_a=x_b$ remains even though it is deleted once, and at most $m$ other vectors get deleted). Hence there are indeed at least $\l-m-1$ distinct vectors among the vectors $x_j$ for $j\in [k]\sm I$.
\end{proof}

On the other hand, our assumption on the set $A$ implies the following structural property for certain solutions $(x_1,\dots,x_k)\in A^k$ to ($\star$) containing an $I$-interesting $I$-tuple.

\begin{claim}\label{claim-structure-interesting}
Let $I\su [k]$ be a subset of size $|I|=m+1$, and suppose that $(x_1,\dots,x_k)\in A^k$ is a solution to  the system ($\star$) such that the $I$-tuple $(x_i\mid i\in I)$ is $I$-interesting and such that there are at least $\l-m-1$ distinct vectors among $x_j$ for $j\in [k]\sm I$. Then at least one of the vectors $x_i$ for $i\in I$ also occurs among the vectors $x_j$ for $j\in [k]\sm I$.
\end{claim}
\begin{proof}
By Definition \ref{defi-interesting}, the vectors $x_i$ for $i\in I$ must be linearly independent (and therefore in particular distinct). Now, $\dim \spn (x_1,\dots,x_k)\geq \dim \spn (x_i\mid i\in I)=m+1$ and by our assumption on the set $A$ this means that there cannot be $\l$ distinct vectors among $x_1,\dots,x_k$. Recall that there are at least $\l-m-1$ distinct vectors among $x_j$ for $j\in [k]\sm I$ and exactly $m+1$ distinct vectors among $x_i$ for $i\in I$. If these two lists of vectors were disjoint, then we would obtain at least $\l$ distinct vectors among $x_1,\dots,x_k$. Hence some vector must occur both among $x_i$ for $i\in I$ and among $x_j$ for $j\in [k]\sm I$.
\end{proof}

Using Claim \ref{claim-structure-interesting}, we can show the following upper bound on the number of $I$-interesting $I$-tuples for any subset $I$. This bound will enable us to perform the desired subspace sampling argument, obtaining a subset $A^*\su A$ without any $I$-interesting $I$-tuples.

\begin{lemma}\label{lemma-number-interesting-tuples}
For each subset $I\su [k]$ of size $|I|=m+1$, there are at most $k^2\cdot p^{mn}$ different $I$-interesting $I$-tuples $(x_i\mid i\in I)$.
\end{lemma}
\begin{proof}
Let us fix a subset $I\su [k]$ of size $|I|=m+1$. For each $I$-interesting $I$-tuples $(x_i\mid i\in I)$, let us consider the sums $\sum_{i\in I}a_{t,i}x_i$ for $t=1,\dots,m$ (in other words, we consider the contributions of the vectors $x_i$ with $i\in I$ to the left-hand sides of the equations in the system ($\star$)). It suffices to prove that for any choice of $b_1,\dots,b_m\in \Fpn$ there can be at most $k^2$ different $I$-interesting $I$-tuples $(x_i\mid i\in I)$ with $\sum_{i\in I}a_{t,i}x_i=b_t$ for $t=1,\dots,m$. Indeed, summing over all $(p^n)^m=p^{mn}$ choices for $b_1,\dots,b_m$ would then give the desired bound on the total number of $I$-interesting $I$-tuples.

So let us fix $b_1,\dots,b_m\in \Fpn$, and suppose that there are more than $k^2$ different $I$-interesting $I$-tuples $(x_i\mid i\in I)$ with $\sum_{i\in I}a_{t,i}x_i=b_t$ for $t=1,\dots,m$. Let $\mathcal{T}$ be the set of all such $I$-interesting $I$-tuples $(x_i\mid i\in I)$, then $|\mathcal{T}|>k^2$.

We claim that for any two distinct $(x_i\mid i\in I),(x_i'\mid i\in I)\in \mathcal{T}$ we must have $x_h\neq x_h'$ for all $h\in I$. Suppose that we had $x_h=x_h'$ for some $h\in I$. Note that then
\[\sum_{i\in I\sm\{h\}}a_{t,i}x_i=b_t-a_{t,h}x_h=  b_t-a_{t,h}x_h'=\sum_{i\in I\sm\{h\}}a_{t,i}x_i'\]
for $t=1,\dots,m$. However, the $m\times m$ matrix $(a_{t,i})_{t\in [m], i\in  I\sm\{h\}}$ is non-singular by the assumption in Theorem \ref{theo-distinct-induction}. Hence we can conclude that $x_i=x_i'$ for all $i\in I\sm\{h\}$ and hence $(x_i\mid i\in I)=(x_i'\mid i\in I)$, which is a contradiction. So for any two distinct $(x_i\mid i\in I),(x_i\mid i\in I)\in \mathcal{T}$ we must indeed have $x_h\neq x_h'$ for all $h\in I$.

Next, we claim that we can find vectors $x_j\in A$ for $j\in [k]\sm I$ such that $\sum_{j\in [k]\sm I}a_{t,j}x_j=-b_t$ for $t=1,\dots,m$ and such that there are at least $\l-m-1$ distinct vectors among $x_j$ for $j\in [k]\sm I$. Indeed, consider any $(x_i\mid i\in I)\in \mathcal{T}$. By Definition \ref{defi-interesting} we can extend $(x_i\mid i\in I)$ to a solution $(x_1,\dots,x_k)\in A^k$ to the system ($\star$) such that there are at least $\l-m-1$ distinct vectors among $x_j$ for $j\in [k]\sm I$. Now, for every $t=1,\dots,m$ we have $\sum_{j\in [k]\sm I}a_{t,j}x_j=-\sum_{i\in I}a_{t,i}x_i=-b_t$, as desired.

So let us now fix a choice of vectors $x_j\in A$ for $j\in [k]\sm I$ such that $\sum_{j\in [k]\sm I}a_{t,j}x_j=-b_t$ for $t=1,\dots,m$ and such that there are at least $\l-m-1$ distinct vectors among $x_j$ for $j\in [k]\sm I$. Then for every $(x_i\mid i\in I)\in \mathcal{T}$, the $k$-tuple $(x_1,\dots,x_k)\in A^k$ satisfies
\[x_{t,1}x_1+\dots+x_{t,k}x_k=\sum_{i\in I}a_{t,i}x_i+\sum_{j\in [k]\sm I}a_{t,j}x_j=b_t+(-b_t)=0\]
for $t=1,\dots,m$. In other words, for every $(x_i\mid i\in I)\in \mathcal{T}$, the $k$-tuple $(x_1,\dots,x_k)\in A^k$ is a solution to the system ($\star$). By Claim \ref{claim-structure-interesting} this means that at least one of the vectors $x_i$ for $i\in I$ must also occur among the vectors $x_j$ for $j\in [k]\sm I$. Hence for every $(x_i\mid i\in I)\in \mathcal{T}$ one of the $\vert I\vert =m+1\leq k$ vectors in the $I$-tuple $(x_i\mid i\in I)$ must be one of the fixed $k-|I|\leq k$ vectors in the set $\{x_j\mid j\in [k]\sm I\}$. As $|\mathcal{T}|>k^2$, by the pigeonhole principle this implies that there must be two distinct $I$-tuples $(x_i\mid i\in I),(x_i'\mid i\in I)\in \mathcal{T}$ with $x_h= x_h'$ for some index $h\in I$. But this is a contradiction to what we showed above. This completes the proof of the lemma.
\end{proof}

Now, let
\[d=\lfloor (1-c)n/m\rfloor,\]
and note that $d\geq 1$ by our assumption that $n\geq m/(1-c)$.

Let us consider a uniformly random $d$-dimensional subspace $V\su \Fpn$. The following two claims give useful bounds for the expected number of vectors in $A\cap V$ and the expected number of $I$-interesting $I$-tuples in $V$.

\begin{claim}\label{claim-expected-vectors-sampling1}
$\EE[|A\cap V|]> (C'+k^22^k)\cdot p^{(1-(m+1)c)n/m}$.
\end{claim}
\begin{proof}
Recall that we assumed $0\not\in A$. So by Lemma \ref{lemma-subspace-probability} (applied with $s=1$), for each vector $x\in A$ we have
\[\PP[x\in V]=\frac{p^d-1}{p^n-1}\geq \frac{p^d/2}{p^n}\geq \frac{1}{2p}\cdot \frac{p^{(1-c)n/m}}{p^n}.\]
Using $\vert A\vert> 2p(C'+k^22^k)\cdot p^{(1-c)n}$, we obtain
\[\EE[|A\cap V|]\geq \frac{1}{2p}\cdot \frac{p^{(1-c)n/m}}{p^n}\cdot |A|> (C'+k^22^k)\cdot p^{((1-c)n/m)-cn}=(C'+k^22^k)\cdot  p^{(1-(m+1)c)n/m},\]
as desired.
\end{proof}

\begin{claim}\label{claim-expected-interesting-tuples}
For each subset $I\su [k]$ of size $|I|=m+1$, the expected number of $I$-interesting $I$-tuples $(x_i\mid i\in I)$ with $x_i\in V$ for all $i\in I$ is at most $k^2\cdot p^{(1-(m+1)c)n/m}$.
\end{claim}
\begin{proof}
Recall from Lemma \ref{lemma-number-interesting-tuples} that the number of $I$-interesting $I$-tuples $(x_i\mid i\in I)$ is at most $k^2\cdot p^{mn}$. For each of these tuples $(x_i\mid i\in I)$ the $m+1$ vectors $x_i$ with $i\in I$ are linearly independent (see Definition \ref{defi-interesting}), so by Lemma \ref{lemma-subspace-probability} we have
\[\PP[x_i\in V\text{ for all }i\in I]\leq \left(\frac{p^d}{p^n}\right)^{m+1}\leq \left(\frac{p^{(1-c)n/m}}{p^n}\right)^{m+1}.\]
Thus, all in all the expected number of $I$-interesting $I$-tuples $(x_i\mid i\in I)$ with $x_i\in V$ for all $i\in I$ is at most
\[\left(\frac{p^{(1-c)n/m}}{p^n}\right)^{m+1}\cdot k^2\cdot p^{mn}=k^2\cdot \frac{p^{(1-c)n(m+1)/m}}{p^n}=k^2\cdot p^{(1-(m+1)c)n/m},\]
as claimed
\end{proof}

Let $Z$ be the total number of $I$-interesting $I$-tuples $(x_i\mid i\in I)$ with $x_i\in V$ for all $i\in I$, summed over all subsets $I\su [k]$ of size $|I|=m+1$. Since there are only $\binom{k}{m+1}\leq 2^k$ choices for $I$, Claim \ref{claim-expected-interesting-tuples} implies that $\EE[Z]\leq 2^k\cdot k^2 \cdot p^{(1-(m+1)c)n/m}$. Hence together with Claim \ref{claim-expected-vectors-sampling1} we obtain
\[\EE[|A\cap V|-Z]> (C'+k^22^k)\cdot p^{(1-(m+1)c)n/m}- 2^k\cdot k^2\cdot  p^{(1-(m+1)c)n/m}= C'\cdot p^{(1-(m+1)c)n/m}.\]
Thus, for the rest of this proof, we can fix an outcome for the random $d$-dimensional subspace  $V\su \Fpn$ for which we have $|A\cap V|-Z> C'\cdot p^{(1-(m+1)c)n/m}$.

We can now define a subset $A^*\su A\cap V$ by deleting one vector  from each $I$-interesting $I$-tuple $(x_i\mid i\in I)$ with $x_i\in V$ for all $i\in I$ (for all subsets $I\su [k]$ of size $|I|=m+1$).  Then
\[|A^*|\geq |A\cap V|-Z> C'\cdot p^{(1-(m+1)c)n/m}\]
and there do not exist any $I$-interesting $I$-tuples $(x_i\mid i\in I)$ with $x_i\in A^*$ for all $i\in I$ (for any $I\su [k]$ of size $\vert I\vert=m+1$).

On the other hand, (\ref{eq-new-old-c-distinct}) yields
\[|A^*|> C'\cdot p^{(1-(m+1)c)n/m}=C'\cdot p^{(1-c')(1-c)n/m}\geq C'\cdot p^{(1-c')d}.\]
As $V\cong \Fp^d$, we can interpret $A^*\su V$ as a subset of $\Fp^d$ of size $|A^*|>C'\cdot p^{(1-c')d}$. By the induction hypothesis (which states that Theorem \ref{theo-distinct-induction} holds for $\l-1$ with the constants $C'$ and $c'$), we can conclude that there must be a solution $(x_1,\dots,x_k)\in (A^*)^k\su A^k$ to the system  ($\star$) such that $\dim\spn(x_1,\dots,x_k)\geq m+1$ and such that among  $x_1,\dots,x_k\in A$ there are at least $\l-1$ distinct vectors. By our assumption on $A$, there must be exactly $\l-1$ distinct vectors among  $x_1,\dots,x_k\in A$. But now Claim \ref{claim-contains-interesting} implies that there exists a subset $I\su [k]$ of size $|I|=m+1$ such that the $I$-tuple $(x_i\mid i\in I)$ is $I$-interesting. Since $x_i\in A^*$ for all $i\in I$, this is a contradiction. This finally finishes the proof of Theorem \ref{theo-distinct-induction}.
\end{proof}

\section{Background on the slice rank polynomial method}
\label{sect-background}

In this section, we will give some background on Tao's slice rank polynomial method, and explain the way in which we will use it in the proof of Theorem \ref{theo-rank}. This involves utiizing a lemma due to Sawin and Tao \cite{sawin-tao} (see Lemma \ref{lemma-sawin-tao} below), which gives lower bounds on the slice rank of non-diagonal tensors of a certain form (see Corollary \ref{coro-sawin-tao} below). This leads to Corollary \ref{coro-slice-rank-partitioned tensor} below, which will then be used as a black box in the proof of Theorem \ref{theo-rank}.

Statements like Lemma \ref{lemma-sawin-tao} and Corollary \ref{coro-sawin-tao} are hard to use in combinatorial applications, since in practice it is often difficult to find a (non-diagonal) tensor with the particular structure required in these statements. In fact, to the author's knowledge, all previous combinatorial applications of the slice rank polynomial method relied on diagonal tensors, meaning that this paper is the first to find a way to exploit the strength of Lemma \ref{lemma-sawin-tao} for non-diagonal tensors. Doing so requires a careful combinatorial setup and analysis, see Section \ref{sect-proof-rank}.

Let us start by defining the notion of slice rank, which was introduced by Tao \cite{tao} in a blog post and later given this name. In our setting it will be most convenient to think of $k$-dimensional tensors as functions $f: [L]^k\to \F$ for some integer $L\geq 0$ and some field $\F$. By considering a vector space $V$ over $\F$ with basis $v_1,\dots,v_L$, one can identify such a function with the element $\sum_{\l_1,\dots,\l_k\in [L]}f(\l_1,\dots,\l_k)\,v_{\l_1}\otimes \dots \otimes v_{\l_k}$ in the $k$-fold tensor product $V\otimes \dots \otimes V$. Yet another way to think of such a function $f: [L]^k\to \F$ is to think of a $k$-dimensional hypermatrix of size $L\times \dots\times L$ with entries in $\F$. However, in our discussion we will formulate everything just in terms of functions $f: [L]^k\to \F$.

\begin{definition}[Tao]
\label{defi-slice-rank}
Let $L\geq 1$ and $k\geq 2$ be integers, and let $\F$ be a field. A function $f:[L]^k\to \F$ has \emph{slice rank 1}, if it can be expressed in the form
\[f(\l_1,\dots,\l_k) = g(\l_i)\cdot h(\l_1,\dots,\l_{i-1},\l_{i+1},\dots,\l_k)\]
for some $i\in [k]$ and non-zero functions $g:[L]\to \F$ and $h:[L]^{k-1}\to \F$. The \emph{slice rank} of an arbitrary function $f:[L]^k\to \F$ is defined to be the minimum number $r$ such that $f$ can be written as the sum of $r$ functions of slice rank 1.
\end{definition}

In other words, a function $f:[L]^k\to \F$ is defined to have slice rank 1, if it can be written as the product of a function depending on just one of the $k$ variables and a function depending on the remaining $k-1$ variables. Note that this notion differs from the standard definition of the rank of a tensor, where a function of rank 1 it is required to be a product of $k$ functions depending on one variable each. The slice rank of a function is always at most as large as its rank according to the standard definition. Also note that the slice rank of any function $f:[L]^k\to \F$ is at most $L$. Indeed, for each $\l\in [L]$, we can take $g_\l$ to be the indicator function of $\l$ and define $h_\l(\l_2,\dots,\l_k)=f(\l,\l_2,\dots,\l_k)$. Then we have $f(\l_1,\dots,\l_k)=\sum_{\l\in [L]} g_\l(\l_i)\cdot h_\l(\l_2,\dots,\l_k)$, which shows that $f:[L]^k\to \F$ has slice rank at most $L$.

Tao's slice rank polynomial method \cite{tao} combines his notion of slice rank as in Definition \ref{defi-slice-rank} with an easy but very powerful polynomial factoring argument. This argument first appeared in work of Croot, Lev, and Pach \cite{croot-lev-pach} on subsets $\Z_4^n$ without $3$-term arithmetic progressions, and was then used again in Ellenberg and Gijswijt's breakthrough \cite{ellenberg-gijswijt} on the cap-set problem bounding the size of $3$-term-progression-free subsets of $\F_3^n$ (and more generally $\F_p^n$ for any fixed $p$). On his blog, Tao \cite{tao} gave a reformulation of the proof of Ellenberg and Gijswijt using the notion of the slice rank (which he introduced in this proof). His reformulation still uses the same polynomial factoring argument that is also at the heart of the proofs of Ellenberg-Gijswijt and of Croot-Lev-Pach in the $\Z_4^n$ setting.

In our context of studying solutions to linear systems of equations, this polynomial factoring argument gives the following lemma. For integers $m\geq 1$ and $k\geq 2m+1$ and a prime $p$, let us define
\begin{equation}\label{eq-defi-Gamma-tao}
\Gamma_{p,m,k}=\min_{0< z\leq 1}\frac{1+z+\dots+z^{p-1}}{z^{(p-1)m/k}}<p.
\end{equation}
It is not hard to see that this minimum exists. Furthermore, at $z=1$ the function on the left-hand size has value $p$ and positive derivative (since $k\geq 2m+1$), which implies that indeed $\Gamma_{p,m,k}<p$. It is also easy to see that $\Gamma_{p,m,k}\geq 1$, so we have $1\leq \Gamma_{p,m,k}<p$.

\begin{lemma}[Croot-Lev-Pach, Tao]
\label{lemma-upper-bound-slice-rank}
Suppose we are given a linear system of equations with coefficients in $\Fp$ and constant terms in $\Fpn$, consisting of $m\geq 1$ equations in $k\geq 2m+1$ variables. Then for any integer $L$ and any vectors $x_i^{(\l)}\in \Fpn$ for $i=1,\dots,k$ and $\l=1,\dots,L$, the function $f:[L]^k\to \Fp$ defined by
\[f(\l_1,\dots,\l_k)=
\begin{cases}
1&\text{if }(x_{1}^{(\l_1)},\dots,x_k^{(\l_k)})\text{ is a solution to the given system of equations}\\
0&\text{otherwise.}
\end{cases}\]
has slice rank at most $k\cdot (\Gamma_{p,m,k})^n$.
\end{lemma}
\begin{proof}
Let the given linear system of equations be of the form 
\begin{align*}
a_{1,1}x_1+\dots+a_{1,k}x_k&=b_1\\
\vdots\qquad\qquad&\\
a_{m,1}x_1+\dots+a_{m,k}x_k&=b_m,
\end{align*}
where $a_j^i\in \Fp$ for $j=1,\dots,m$ and $i=1,\dots,k$ and $b_j\in \Fpn$ for $j=1,\dots,m$. Furthermore, for any vector $x\in \Fpn$, let us write $x(1),\dots, x(n)$ for the coordinates of $x$ (which are elements of $\Fp$). In particular, $x_i^{(\l)}(1),\dots, x_i^{(\l)}(n)$ are the coordinates of $x_i^{(\l)}\in\Fpn$ for $i=1,\dots,k$ and $\l=1,\dots,L$ and $b_j(1),\dots, b_j(n)$ are the coordinates of $b_j\in\Fpn$ for $j=1,\dots,m$.

Now, we claim that
\begin{equation}\label{eq-polynomial-tensor}
f(\l_1,\dots,\l_k)=\prod_{j=1}^{m}\prod_{s=1}^{n}\left(1-(a_{j,1}x_1^{(\l_1)}(s)+\dots+a_{j,k}x_k^{(\l_k)}(s)-b_j(s))^{p-1}\right)
\end{equation}
for all $\l_1,\dots,\l_k\in [L]$. Indeed, if $\l_1,\dots,\l_k\in [L]$ are such that $(x_{1}^{(\l_1)},\dots,x_k^{(\l_k)})$ is a solution to the system of equations above, then we have $a_{j,1}x_1^{(\l_1)}+\dots+a_{j,k}x_k^{(\l_k)}-b_j=0$ for all $j=1,\dots,m$ and consequently $a_{j,1}x_1^{(\l_1)}(s)+\dots+a_{j,k}x_k^{(\l_k)}(s)-b_j(s)=0$ for all $j=1,\dots,m$ and $s=1,\dots,n$. This means that all factors on the right-hand side of (\ref{eq-polynomial-tensor}) are equal to 1, and therefore the product is indeed equal to $f(\l_1,\dots,\l_k)=1$. In the other case, where $(x_{1}^{(\l_1)},\dots,x_k^{(\l_k)})$ is not a solution to the system of equations, there must be some $j\in\{1,\dots,m\}$ and $s\in \{1,\dots,n\}$ such that $a_{j,1}x_1^{(\l_1)}(s)+\dots+a_{j,k}x_k^{(\l_k)}(s)-b_j(s)\neq 0$. But then we have $(a_{j,1}x_1^{(\l_1)}(s)+\dots+a_{j,k}x_k^{(\l_k)}(s)-b_j(s))^{p-1}=1$, and so the factor $1-(a_{j,1}x_1^{(\l_1)}(s)+\dots+a_{j,k}x_k^{(\l_k)}(s)-b_j(s))^{p-1}$ on the right-hand side of (\ref{eq-polynomial-tensor}) is $0$. Hence the entire product on the right-hand side is $0$ and therefore equal to $f(\l_1,\dots,\l_k)=0$. Thus, (\ref{eq-polynomial-tensor})  is indeed true.

We can now use the polynomial representation of $f$ in (\ref{eq-polynomial-tensor}) to show the desired upper bound on the slice rank of $f$. Note that the right-hand side of  (\ref{eq-polynomial-tensor}) is a polynomial of degree $mn(p-1)$ in the $kn$ variables $x_i^{(\l_i)}(1),\dots,x_i^{(\l_i)}(n)$ for $i=1,\dots,k$. Let us imagine that we multiply out the product on the right-hand side of  (\ref{eq-polynomial-tensor}). Then we can write $f(\l_1,\dots,\l_k)$ as a linear combination of monomials in the variables $x_i^{(\l_i)}(1),\dots,x_i^{(\l_i)}(n)$ for $i=1,\dots,k$, where each monomial has degree at most $mn(p-1)$.

Since in $\Fp$ we have $y^p=y$ for all $y\in \Fp$, we can replace each higher power $(x_i^{(\l_i)}(s))^d$ with $d\geq p$ by a power $(x_i^{(\l_i)}(s))^{d'}$ with $d'\in \{1,\dots,p-1\}$ (and $d'\equiv d\mod{p-1}$). This way we can represent $f(\l_1,\dots,\l_k)$ as a linear combination of monomials in the variables $x_i^{(\l_i)}(1),\dots,x_i^{(\l_i)}(n)$ for $i=1,\dots,k$, where each monomial has degree at most $mn(p-1)$ and each individual variable appears with degree at most $p-1$.

For each of the monomials in this representation the total degree is at most $mn(p-1)$, so there must be some $i\in [k]$ such that the monomial has degree at most $mn(p-1)/k$ in the variables $x_i^{(\l_i)}(1),\dots,x_i^{(\l_i)}(n)$. Hence each monomial is of the form
\[(x_i^{(\l_i)}(1))^{d_1}\dotsm (x_i^{(\l_i)}(n))^{d_n}\cdot g(\l_1,\dots,\l_{i-1},\l_{i+1},\dots,\l_k)\]
for some $i\in [k]$, some $d_1,\dots,d_n\in \{0,\dots, p-1\}$ with $d_1+\dots+d_n\leq mn(p-1)/k$ and some function $g:[L]^{k-1}\to \Fp$ (where $g$ is a monomial in the variables $x_{i'}^{(\l_{i'})}(1),\dots,x_{i'}^{(\l_{i'})}(n)$ for $i\in [k]\sm \{i\}$).

By grouping together the monomials with the same $i$ and the same $d_1,\dots,d_n$ in the above representation, we now obtain a representation of $f(\l_1,\dots,\l_k)$ as a sum of terms of the form
\[(x_i^{(\l_i)}(1))^{d_1}\dotsm (x_i^{(\l_i)}(n))^{d_n}\cdot h(\l_1,\dots,\l_{i-1},\l_{i+1},\dots,\l_k)\]
for some $i\in [k]$, some $d_1,\dots,d_n\in \{0,\dots, p-1\}$ with $d_1+\dots+d_n\leq mn(p-1)/k$ and some function $h:[L]^{k-1}\to \Fp$ (here, the function $h$ is obtained as a linear combination of the functions $g$ that we previously considered). Note that each such term is a slice rank 1 function. Hence the slice rank of $f$ is at most
\[k\cdot |\{(d_1,\dots,d_n)\in \{0,\dots, p-1\}^n\mid d_1+\dots+d_n\leq mn(p-1)/k\}|\]
Thus, using the following claim, we obtain desired bound for the slice rank of $f$ .

\begin{claim}\label{claim-bound-number-tuples}
$|\{(d_1,\dots,d_n)\in \{0,\dots, p-1\}^n\mid d_1+\dots+d_n\leq mn(p-1)/k\}|\leq (\Gamma_{p,m,k})^n$.
\end{claim}
\begin{proof}
We need to prove that for a uniformly random choice of $(d_1,\dots,d_n)\in \{0,\dots, p-1\}^n$ we have $\PP[d_1+\dots+d_n\leq mn(p-1)/k]\leq (\Gamma_{p,m,k})^n\cdot p^{-n}$. Indeed, for any $0<z\leq 1$, by Markov's inequality we have
\begin{multline*}
\PP[d_1+\dots+d_n\leq mn(p-1)/k]\leq \PP[z^{d_1+\dots+d_n}\geq z^{mn(p-1)/k}]\\
\leq \frac{\EE[z^{d_1+\dots+d_n}]}{z^{mn(p-1)/k}}=\frac{\EE[z^{d_1}]\dotsm \EE[z^{d_n}]}{z^{mn(p-1)/k}}
=\frac{((1+z+\dots+z^{p-1})/p)^n}{z^{mn(p-1)/k}}=\left(\frac{1+z+\dots+z^{p-1}}{z^{(p-1)m/k}}\right)^n\cdot p^{-n},
\end{multline*}
where we used that $d_1,\dots,d_n$ can be viewed as independent  uniformly random elements of $\{0,\dots, p-1\}$. Hence
\[\PP[d_1+\dots+d_n\leq mn(p-1)/k]\leq \left(\min_{0<z\leq 1}\frac{1+z+\dots+z^{p-1}}{z^{(p-1)m/k}}\right)^n\cdot p^{-n}=(\Gamma_{p,m,k})^n\cdot p^{-n},\]
as desired.
\end{proof}
This finishes the proof of Lemma \ref{lemma-upper-bound-slice-rank}.
\end{proof}

We remark that the proof of Lemma \ref{lemma-upper-bound-slice-rank} is a straightforward generalization of the corresponding arguments in Tao's blog post \cite{tao} (which are for the case $m=1$ and $k=3$), apart from the bound for the quantity in Claim \ref{claim-bound-number-tuples}. This bound as well as the proof of Claim \ref{claim-bound-number-tuples} appeared for example in \cite[Proposition 4.12]{blasiak-et-al}.

As an example of a typical application of Tao's slice rank polynomial method, let us now show how Theorem \ref{theo-tao} can be deduced from Lemma \ref{lemma-upper-bound-slice-rank} and the following lemma due to Tao \cite[Lemma 1]{tao} stating that diagonal tensors have large slice rank.

\begin{lemma}[Tao]
\label{lemma-slice-rank-diag-tensor}
Let $L\geq 1$ and $k\geq 2$ be integers, and let $\F$ be a field. Suppose that $f:[L]^k\to \F$ is a function such that $f(\l_1,\dots,\l_k)\neq 0$ whenever $\l_1=\dots=\l_k$ and such that $f(\l_1,\dots,\l_k)= 0$ whenever $\l_1,\dots,\l_k\in [L]$ are not all equal. Then the slice rank of $f$ is equal to $L$.
\end{lemma}

By combining Lemmas \ref{lemma-upper-bound-slice-rank} and \ref{lemma-slice-rank-diag-tensor} one can obtain Theorem \ref{theo-tao}. This proof appeared in Tao's blog post \cite{tao} in the special case of $m=1$, $k=3$ and ($\star$) being the single equation $x_1-2x_2+x_3=0$, which corresponds to $3$-term arithmetic progressions (in the same blog post he also introduced the notion of slice rank and proved Lemma \ref{lemma-slice-rank-diag-tensor}).

\begin{proof}[Proof Theorem \ref{theo-tao}]
Let $A\su \Fpn$ be such that every solution $(x_1,\dots,x_k)\in A^k$ of the system ($\star$) satisfies $x_1=\dots=x_k$. Let $L=|A|$, and  let $A=\{x^{(1)},\dots, x^{(L)}\}$. Now, define the function $f:[L]^k\to \Fp$ by setting
\[f(\l_1,\dots,\l_k)=
\begin{cases}
1&\text{if }(x^{(\l_1)},\dots,x^{(\l_k)})\text{ is a solution to }(\star)\\
0&\text{otherwise.}
\end{cases}.\]
Note that by Lemma \ref{lemma-upper-bound-slice-rank} the slice rank of $f$ is at most $k\cdot (\Gamma_{p,m,k})^n$.

On the other hand, whenever $\l_1,\dots,\l_k$ are such that $(x^{(\l_1)},\dots,x^{(\l_k)})$ is a solution to  ($\star$), then by our assumption on $A$ we have $x^{(\l_1)}=\dots=x^{(\l_k)}$ and therefore $\l_1=\dots=\l_k$. Thus, we have $f(\l_1,\dots,\l_k)= 0$ whenever $\l_1,\dots,\l_k\in [L]$ are not all equal. Furthermore, for all $\l\in L$, the $k$-tuple $(x^{(\l)},\dots,x^{(\l)})$ is a solution to ($\star$), since we assumed that $a_{j,1}+\dots+a_{j,k}=0$ for $j=1,\dots,m$. Hence whenever $\l_1=\dots=\l_k$, we have $f(\l_1,\dots,\l_k)= 1\neq 0$. Thus, by Lemma \ref{lemma-slice-rank-diag-tensor} the slice rank of $f$ equals $L$.

All in all, we obtain $|A|=L\leq k\cdot (\Gamma_{p,m,k})^n$. So we have shown that every set $A\su \Fpn$ with the property that every solution $(x_1,\dots,x_k)\in A^k$ of the system ($\star$) satisfies $x_1=\dots=x_k$ must have size $|A|\leq k\cdot (\Gamma_{p,m,k})^n$.

This almost gives Theorem \ref{theo-tao}. In order to prove the theorem, we need to prove $|A|\leq (\Gamma_{p,m,k})^n$ instead of the weaker bound $|A|\leq k\cdot (\Gamma_{p,m,k})^n$. We can use a power trick, as follows.

Suppose that $A\su \Fpn$ is a set with the property that every solution $(x_1,\dots,x_k)\in A^k$ of the system ($\star$) satisfies $x_1=\dots=x_k$. Note that for every positive integer $m$, the set  $A^m=A\times\dots\times A\su \Fpn\times \dots\times \Fpn=\Fp^{nm}$ also has this property. Therefore, by what we proved above, we must have $|A|^m=|A^m|\leq k\cdot (\Gamma_{p,m,k})^{nm}$ and therefore  $|A|\leq k^{1/m}\cdot (\Gamma_{p,m,k})^n$ for every positive integer $m$. Hence  $|A|\leq (\Gamma_{p,m,k})^n$, as desired.
\end{proof} 

The proof of Theorem \ref{theo-tao} follows the typical pattern of applications of the slice rank polynomial method in combinatorics: If one  wants to bound the size of a combinatorial structure satisfying certain conditions (here the set $A=\{x^{(1)},\dots,x^{(L)}\}$), one defines a function $f$ in terms of this structure in such a way that a polynomial factoring argument as in the proof of Lemma \ref{lemma-upper-bound-slice-rank} gives an upper bound on the slice rank of $f$. On the other hand, the conditions on the combinatorial structure imply that the function $f$ must be of a special form, giving a lower bound for its slice rank in terms of the size of the structure. Combining both bounds, one then obtains an upper bound for the size of the combinatorial structure, as desired.

As mentioned above, so far in all combinatorial applications of the slice rank polynomial method, the ``special form'' of the function $f$ under consideration was a diagonal tensor as in Lemma \ref{lemma-slice-rank-diag-tensor}. Not long after Tao's original blog post \cite{tao}, another post by Sawin and Tao \cite{sawin-tao} appeared on Tao's blog,  proving the following statement \cite[Proposition 4]{sawin-tao} concerning the slice rank of certain non-diagonal tensors.

\begin{lemma}[Sawin-Tao]
\label{lemma-sawin-tao}
Let $L\geq 1$ and $k\geq 2$ be integers, and let $\F$ be a field. Fix $k$ total orderings $\preceq^1,\dots, \preceq^k$ on the set $[L]=\{1,\dots,L\}$, and consider the resulting product partial order $\preceq$ on $[L]^k$. Suppose that $f:[L]^k\to \F$ is a function such that the set $S=\{(\l_1,\dots,\l_k)\in [L]^k\mid f(\l_1,\dots,\l_k)\neq 0\}\su [L]^k$ is an antichain with respect to the partial order $\preceq$. Then the slice rank of $f$ equals
\[\min_{S=S_1\cup \dots\cup S_k}\left(|\pi_1(S_1)|+\dots+|\pi_k(S_k)|\right),\]
where the minimum is taken over all partitions $S=S_1\cup \dots\cup S_k$ and where $\pi_i:[L]^k\to L$ denotes the projection to the $i$-th factor for $i=1,\dots,k$.
\end{lemma}

The following corollary of Lemma \ref{lemma-sawin-tao} gives a slightly more concrete statement about certain tensors having large slice rank.

\begin{corollary}\label{coro-sawin-tao}
Let $L\geq 1$ and $k\geq 2$ be integers, and let $\F$ be a field. Let $[k]=J_1\cup \dots\cup J_t$ be a partition such that $|J_h|\geq 2$ for $h=1,\dots,t$. Now,  suppose that $f:[L]^k\to \F$ is a function such that $f(\l,\dots,\l)\neq 0$ for all $\l\in [L]$ and such that for any choice of $\l_1,\dots\l_k\in [L]$ with $f(\l_1,\dots,\l_k)\neq 0$ we have $|\{ \l_j \mid j\in J_h\}|=1$ for $h=1,\dots,t$. Then the slice rank of $f$ is equal to $L$.
\end{corollary}
\begin{proof}
Let us define total orderings $\preceq^1,\dots, \preceq^k$ on the set $[L]$ as follows. For each $h=1,\dots,t$, choose one element $j_h\in J_h$ and define $\preceq^{j_h}$ to be the canonical increasing ordering $1\preceq^{j_h} 2\preceq^{j_h}\dots \preceq^{j_h} L$ on $[L]$. Furthermore, choose a different element $i_h\in J_h$ (recall that $|J_h|\geq 2$) and define the total ordering $\preceq^{i_h}$ to be the opposite total ordering on $[L]$, i.e.\ the ordering $L\preceq^{i_h} L-1\preceq^{i_h}\dots \preceq^{i_h} 1$. For all remaining elements $j\in J_h\sm \{j_h,i_h\}$, define $\preceq^j$ to be an arbitrary total ordering on $[L]$.

We claim that the function $f:[L]^k\to \F$ satisfies the assumption in Lemma \ref{lemma-sawin-tao} with respect to the total orderings $\preceq^1,\dots, \preceq^k$ we just defined. Indeed, suppose that $(\l_1,\dots,\l_k), (\l_1',\dots,\l_k')\in [L]^k$ with $f(\l_1,\dots,\l_k)\neq 0$ and $f(\l_1',\dots,\l_k')\neq 0$ are such that $(\l_1,\dots,\l_k)\preceq (\l_1',\dots,\l_k')$ in the product partial order $\preceq$. Then $\l_i\preceq^i \l_i'$ for all $i\in [k]$. We will show that we actually have $\l_i=\l_i'$ for all $i\in [k]$. So let $i\in [k]$, and let $h\in [t]$ be such that $i\in J_h$. By our assumption on $f$ we have $|\{ \l_j \mid j\in J_h\}|=1$ and $|\{ \l_j' \mid j\in J_h\}|=1$, which means that $\l_i=\l_{j_h}=\l_{i_h}$ and $\l_i'=\l_{j_h}'=\l_{i_h}'$ for the elements $j_h,i_h\in J_h$ that we fixed above. Hence $\l_i=\l_{j_h}\preceq^{j_h} \l_{j_h}'=\l_i'$ in the total ordering  $\preceq^{j_h}$, which by the definition of $\preceq^{j_h}$ means that $\l_i\leq \l_i'$. On the other hand we have $\l_i=\l_{i_h}\preceq^{i_h} \l_{i_h}'=\l_i'$ in the total ordering  $\preceq^{i_h}$, which by the definition of $\preceq^{i_h}$ means that $\l_i\geq \l_i'$. Thus, we can conclude that $\l_i=\l_i'$ for all $i\in [k]$, so $(\l_1,\dots,\l_k)=(\l_1',\dots,\l_k')$.

Hence the set $S=\{(\l_1,\dots,\l_k)\in [L]^k\mid f(\l_1,\dots,\l_k)\neq 0\}\su [L]^k$ is indeed an antichain with respect to the partial order $\preceq$. So Lemma \ref{lemma-sawin-tao} applies and we can conclude that the slice rank of $f$ equals
\[\min_{S=S_1\cup \dots\cup S_k}\left(|\pi_1(S_1)|+\dots+|\pi_k(S_k)|\right),\]
where the minimum is taken over all partitions $S=S_1\cup \dots\cup S_k$.

We claim that for any partition $S=S_1\cup \dots\cup S_k$ we must have $\pi_1(S_1)\cup\dots \cup\pi_k(S_k)=[L]$. Indeed, for any $\l\in [L]$, by the assumption $f(\l,\dots,\l)\neq 0$ we have $(\l,\dots,\l)\in S$ and therefore $(\l,\dots,\l)\in S_i$ for some $i\in [k]$. This means that $\l\in \pi_i(S_i)\su \pi_1(S_1)\cup\dots \cup\pi_k(S_k)$. Hence $\pi_1(S_1)\cup\dots \cup\pi_k(S_k)=[L]$ and therefore $|\pi_1(S_1)|+\dots+|\pi_k(S_k)|\geq |\pi_1(S_1)\cup\dots \cup\pi_k(S_k)|=L$ for any partition  $S=S_1\cup \dots\cup S_k$.

Thus, the slice rank of $f$ is at least $L$. On the other hand, the slice rank of $f:[L]^k\to \F$ is also at most $L$. Hence the slice rank of $f$ equals $L$.
\end{proof}

Note that Corollary \ref{coro-sawin-tao} can be viewed as a generalization of Lemma \ref{lemma-slice-rank-diag-tensor} for diagonal tensors by considering the partition of $[k]$ into a single set $J_1=[k]$. The slice rank polynomial method has had many interesting applications in combinatorics (see for example Grochow's survey \cite{grochow}), but all of them rely on using a diagonal tensor (meaning that they proceed via Lemma \ref{lemma-slice-rank-diag-tensor}). This paper gives the first combinatorial application that uses the additional strength of Corollary \ref{coro-sawin-tao} for non-diagonal tensors.

Being able to use Corollary \ref{coro-sawin-tao} in the proof of Theorem \ref{theo-rank} requires new combinatorial ideas and a rather technical inductive setup. In order to not further complicate the presentation of the proof of Theorem \ref{theo-rank} in Section \ref{sect-proof-rank}, we record the following statement here, which combines Corollary \ref{coro-sawin-tao} with Lemma \ref{lemma-upper-bound-slice-rank}. This statement will then be used in the proof of Theorem \ref{theo-rank} as black box, so that no discussion of slice rank arguments is required in Section \ref{sect-proof-rank} anymore.

\begin{corollary}\label{coro-slice-rank-partitioned tensor}
Suppose we are given a linear system of equations with coefficients in $\Fp$ and constant terms in $\Fpn$, consisting of $m\geq 1$ equations in $k\geq 2m+1$ variables. Let $(x_1^{(\l)},\dots,x_k^{(\l)})\in (\Fpn)^k$ for $\l=1,\dots,L$ be solutions in $\Fpn$ to this system of equations. Suppose that there exists a partition $[k]=J_1\cup \dots\cup J_t$ with $|J_h|\geq 2$ for $h=1,\dots,t$ such that the following condition holds: For any choice of $\l_1,\dots\l_k\in [L]$ such that $(x_1^{(\l_1)},\dots,x_k^{(\l_k)})$ is a solution to the given system of equations, we have $|\{ \l_j \mid j\in J_h\}|=1$ for all $h=1,\dots,t$. Then we must have $L\leq k\cdot (\Gamma_{p,m,k})^n$.
\end{corollary}
\begin{proof}
Let us define the function $f:[L]^k\to \Fp$ as in Lemma \ref{lemma-upper-bound-slice-rank}, then the slice rank of $f$ is at most $k\cdot (\Gamma_{p,m,k})^n$. On the other hand, the function $f$ satisfies the assumptions in Corollary \ref{coro-sawin-tao}, so the slice rank of $f$ is equal to $L$. Hence $L\leq k\cdot (\Gamma_{p,m,k})^n$, as desired.
\end{proof}

\begin{remark} We remark hat the constant factor $k$ can be removed from the bound $L\leq k\cdot (\Gamma_{p,m,k})^n$ in Corollary \ref{coro-slice-rank-partitioned tensor} by a power trick similar to the one in the proof of Theorem \ref{theo-tao}. However, since this power trick is notationally cumbersome in this setting and since the bound $L\leq k\cdot (\Gamma_{p,m,k})^n$ suffices for our purposes, we stated Corollary \ref{coro-slice-rank-partitioned tensor} with this weaker bound.
\end{remark}

\section{Proof of Theorem \ref{theo-rank}}
\label{sect-proof-rank}

\subsection{Inductive setup}

Somewhat similar to our approach in Section \ref{sect-proof-distinct}, we will also an inductive argument to prove Theorem \ref{theo-rank}. However, this induction requires a somewhat technical setup. We will associate a certain weight to each solution $(x_1,\dots,x_k)$ to the system ($\star$), and then induct on this quantity.

Let us fix positive integers $m$ and $k$, a prime $p$ and coefficients $a_{j,i}\in \Fp$ for the system ($\star$) such that the assumption $a_{j,1}+\dots+a_{j,k}=0$ for $j=1,\dots,m$ in Theorem \ref{theo-rank} is satisfied. We start by making the following definitions.

\begin{definition}\label{defi-admissible}
Let $V$ be a vector space over $\Fp$. For a given solution $(x_1,\dots,x_k)\in (V\sm \{0\})^k$ to the system ($\star$), let us say that a subset $I\su [k]$ is \emph{admissible} if it satisfies the following two conditions:
\begin{itemize}
\item[(i)] The vectors $x_i$ for $i\in I$ are linearly independent.
\item[(ii)] For all $j\in [k]\sm I$ we have $x_j\not\in \spn(x_i\mid i\in I)$.
\end{itemize}
\end{definition}

Note that the empty set $I=\emptyset$ always satisfies the conditions for being admissible (here, we use that the vectors $x_1,\dots,x_k$ are all non-zero). In particular, for every solution $(x_1,\dots,x_k)\in (V\sm \{0\})^k$ to the system ($\star$) there exists some admissible subset $I\su [k]$.

\begin{definition}\label{defi-weight-set}
Let $V$ be a vector space over $\Fp$. For a given solution $(x_1,\dots,x_k)\in (V\sm \{0\})^k$ to the system ($\star$), let us define the \emph{weight} of an admissible subset $I\su [k]$ as follows: Letting $U=\spn(x_i\mid i\in I)$, we define the weight of $I$ with respect to $(x_1,\dots,x_k)$ to be
\[(k+1)\cdot |I|+\left|\left\{\spn(\proj_{V/U}(x_j))\,\Big|\, j\in [k]\sm I\right\}\right|.\]
\end{definition}

The second summand in the expression above for the weight of $I$ counts the number of different subspaces of the quotient space $V/U$ that are of the form $\spn(\proj_{V/U}(x_j))$ for some $j\in [k]\sm I$. Note that we clearly have $|\{\spn(\proj_{V/U}(x_j))\mid j\in [k]\sm I\}|\leq k-|I|\leq k$. 

\begin{definition}\label{defi-weight-solution}
Let $V$ be a vector space over $\Fp$. For a solution $(x_1,\dots,x_k)\in (V\sm \{0\})^k$ to the system ($\star$), let us define the \emph{weight} $\omega(x_1,\dots,x_k)$ of $(x_1,\dots,x_k)$ to be the maximum weight of any admissible subset $I\su [k]$ with respect to $(x_1,\dots,x_k)$. Furthermore, let us define $\mathcal{I}(x_1,\dots,x_k)\su [k]$ to be an admissible subset $I\su [k]$ for $(x_1,\dots,x_k)$ where this maximum weight is attained (if there are several choices for $I$ attaining the maximum weight, we choose one arbitrarily).
\end{definition}

Note that the weight $\omega(x_1,\dots,x_n)$ is clearly a non-negative integer. Also note that for a vector spaces $V\su V'$ over $\Fp$, and a solution $(x_1,\dots,x_k)\in (V\sm \{0\})^k\su (V'\sm \{0\})^k$ to the system ($\star$), the weight $\omega(x_1,\dots,x_k)$ of $(x_1,\dots,x_k)$ does not depend on whether $x_1,\dots,x_k$ are interpreted as vectors in $V$ or in $V'$.

\begin{claim}\label{claim-properties-weight}
Let $V$ be a vector space over $\Fp$, and let $(x_1,\dots,x_k)\in (V\sm \{0\})^k$ be a solution to the system ($\star$). Then we have $\omega(x_1,\dots,x_k)\not\in \{0,(k+1),2\cdot (k+1),\dots,(k-1)\cdot (k+1)\}$,
\[|\mathcal{I}(x_1,\dots,x_k)|= \lfloor\omega(x_1,\dots,x_k)/(k+1)\rfloor,\]
and
\[\dim\spn(x_1,\dots,x_k)\geq \omega(x_1,\dots,x_k)/(k+1).\]
\end{claim}
\begin{proof}
Let $I=\mathcal{I}(x_1,\dots,x_k)$, and let $U=\spn(x_i\mid i\in I)$. Then  $\omega(x_1,\dots,x_k)$ equals the weight of $I$ with respect to $(x_1,\dots,x_k)$. In other words,
\begin{equation}\label{eq-expression-omega}
\omega(x_1,\dots,x_k)=(k+1)\cdot |I|+\left|\left\{\spn(\proj_{V/U}(x_j))\,\Big|\, j\in [k]\sm I\right\}\right|.
\end{equation}

First, let us consider the case $I=[k]$. Then we have $|\mathcal{I}(x_1,\dots,x_k)|=|I|=k$ and $\omega(x_1,\dots,x_k)=(k+1)\cdot k$, and $\dim\spn(x_1,\dots,x_k)\geq \dim \spn(x_i\mid i\in I)=\vert I\vert=k$ (here, we used condition (i) in Definition \ref{defi-admissible}). Hence the statements in the claim are satisfied in the case $I=[k]$.

Now let us assume that $I\neq [k]$. Then there is at least one vector $x_j$ with $j\in [k]\sm I$, and so from (\ref{eq-expression-omega}) we obtain
\[(k+1)\cdot |I|+1\leq \omega(x_1,\dots,x_k)\leq (k+1)\cdot |I|+k.\]
This in particular implies $ \lfloor\omega(x_1,\dots,x_k)/(k+1)\rfloor=|I|=|\mathcal{I}(x_1,\dots,x_k)|$. It furthermore implies that $\omega(x_1,\dots,x_k)$ is not divisible by $k+1$ and hence $\omega(x_1,\dots,x_k)\not\in \{0,(k+1),2\cdot (k+1),\dots,(k-1)\cdot (k+1)\}$. 

Finally, recall from condition (i) in Definition \ref{defi-admissible} that the vectors $x_i$ for $i\in I$ are linearly independent. Furthermore, there is some $j\in [k]\sm I$ and by condition (ii) in Definition~\ref{defi-admissible} we have $x_j\not\in \spn(x_i\mid i\in I)$ . Thus, $\dim\spn(x_1,\dots,x_k)\geq \vert I\vert+1\geq \omega(x_1,\dots,x_k)/(k+1)$.
\end{proof}

The following proposition is similar to Theorem \ref{theo-rank}, but instead of finding a solution $(x_1,\dots,x_k)\in A^k$ to ($\star$) with $\dim\spn(x_1,\dots,x_k)\geq r$, it aims to find a solution $(x_1,\dots,x_k)\in A^k$ with large weight $\omega(x_1,\dots,x_k)$. We will prove Proposition \ref{propo-weight} by induction on $w$, and we will deduce Theorem \ref{theo-rank} by considering $w=(r-1)(k+1)-1$. Recall that we are assuming $a_{j,1}+\dots+a_{j,k}=0$ for $j=1,\dots,m$.

\begin{proposition}\label{propo-weight}
Fix a non-negative integer $w$ and assume that $k\geq 2m+1+\lfloor w/(k+1)\rfloor$. Then there exist constants $C\geq 1$ and $0<c\leq 1$ such that the following holds: For any non-negative integer $n$ and any subset $A\su \Fpn\sm\{0\}$ of size $|A|> C\cdot p^{(1-c)n}$, the system  ($\star$) has a solution $(x_1,\dots,x_k)\in A^k$ with weight $\omega(x_1,\dots,x_k)> w$.
\end{proposition}

Let us now deduce Theorem \ref{theo-rank} from Proposition \ref{propo-weight}.

\begin{proof}[Proof of Theorem \ref{theo-rank}]
As in the theorem statement, let $r\geq 2$ and assume that $k\geq 2m+r-1$. Now, let $w=(r-1)(k+1)-1$, and note that then $\lfloor w/(k+1)\rfloor=r-2$ and therefore $k\geq 2m+1+r-2=2m+1+\lfloor w/(k+1)\rfloor$. Let us define $C_{p,m,k,r}^{\textnormal{rank}}=C+1$ and $\Gamma_{p,m,k,r}^{\textnormal{rank}}=p^{1-c}<p$, where $C\geq 1$ and $0<c\leq 1$ are the constants obtained from Proposition \ref{propo-weight}.

Now, let $n$ be a non-negative integer and let $A\su \Fpn$ be a subset of size $|A|> C_{p,m,k,r}^{\textnormal{rank}}\cdot (\Gamma_{p,m,k,r}^{\textnormal{rank}})^n$. Then we have
\[|A\sm\{0\}|\geq |A|-1> (C_{p,m,k,r}^{\textnormal{rank}}-1)\cdot (\Gamma_{p,m,k,r}^{\textnormal{rank}})^n=C\cdot p^{(1-c)n}.\]
Thus, by Proposition \ref{propo-weight} applied to $A\sm\{0\}$, there exists a solution $(x_1,\dots,x_k)\in (A\sm\{0\})^k\su A^k$ to the system ($\star$) with weight $\omega(x_1,\dots,x_k)> w=(r-1)(k+1)-1$, Note that this means that $\omega(x_1,\dots,x_k)\geq (r-1)(k+1)$. However, as $2\leq r\leq k-1$, the first statement in Claim \ref{claim-properties-weight} implies that $\omega(x_1,\dots,x_k)\neq (r-1)(k+1)$. Hence $\omega(x_1,\dots,x_k)> (r-1)(k+1)$ and consequently by the last statement in Claim \ref{claim-properties-weight} we have $\dim\spn(x_1,\dots,x_k)\geq \omega(x_1,\dots,x_k)/(k+1)>r-1$. This means that $\dim\spn(x_1,\dots,x_k)\geq r$, as desired.
\end{proof}

We will prove Proposition \ref{propo-weight} in the following subsection. In the proof, we will use the following lemma, giving a structural property for the spans $\spn(\proj_{V/U}(x_j))$ appearing in Definition \ref{defi-weight-set}.

\begin{lemma}\label{lemma-structure-partition-exists}
Let $V$ be a vector space over $\Fp$, and let $(x_1,\dots,x_k)\in (V\sm \{0\})^k$ be a solution to the system ($\star$). Furthermore, let $I=\mathcal{I}(x_1,\dots,x_k)$ and $U=\spn(x_i\mid i\in I)$. Then there exists a partition $[k]\sm I=J_1\cup \dots\cup J_t$ with $|J_h|\geq 2$ for $h=1,\dots,t$ and distinct one-dimensional subspaces $W_1,\dots,W_t\su V/U$ such that $\spn(\proj_{V/U}(x_j))=W_h$ for all $h=1,\dots,t$  and all $j\in J_h$.
\end{lemma}
\begin{proof}
Recall that the set $I=\mathcal{I}(x_1,\dots,x_k)$ is admissible for $(x_1,\dots,x_k)$, and that therefore by condition (ii) in Definition \ref{defi-admissible} we have $x_j\not\in \spn(x_i\mid i\in I)=U$ for all $j\in [k]\sm I$. Hence for every $j\in [k]\sm I$, the projection $\proj_{V/U}(x_j)$ is a non-zero vector in $V/U$ and therefore $\spn(\proj_{V/U}(x_j))$ is a one-dimensional subspace of $V/U$. 

Let $W_1,\dots,W_t\su V/U$ be the list of all distinct one-dimensional subspaces of $V/U$ that are of the form $\spn(\proj_{V/U}(x_j))$ for some $j\in [k]\sm I$. Furthermore, for each $h=1,\dots,t$, let $J_h\su [k]\sm I$ be the set of indices $j\in [k]\sm I$ such that $\spn(\proj_{V/U}(x_j))=W_h$. Then $[k]\sm I=J_1\cup \dots\cup J_t$ is a partition.

It remains to show that we have $|J_h|\geq 2$ for $h=1,\dots,t$. Note that this is equivalent to showing that for each $j\in [k]\sm I$ there exists some $j'\in [k]\sm I$ with $j'\neq j$ and $\spn(\proj_{V/U}(x_j))=\spn(\proj_{V/U}(x_{j'}))$.

So let us fix some $j\in [k]\sm I$, and suppose for contradiction that $\spn(\proj_{V/U}(x_{j'}))\neq \spn(\proj_{V/U}(x_{j}))$ for all $j'\in [k]\sm (I\cup \{j\})$. As both $\spn(\proj_{V/U}(x_{j'}))$ and $\spn(\proj_{V/U}(x_{j}))$ are one-dimensional subspaces of $V/U$, this means that $\proj_{V/U}(x_{j'})\not\in \spn(\proj_{V/U}(x_{j}))$ for all $j'\in [k]\sm (I\cup \{j\})$. In other words, we have $x_{j'}\not\in U+\spn(x_j)=\spn(x_i\mid i\in I\cup \{j\})$ for all $j'\in [k]\sm (I\cup \{j\})$. Hence the set $I\cup \{j\}$ satisfies condition (ii) in Definition \ref{defi-admissible}.

Furthermore, recall that the vectors $x_i$ for $i\in I$ are linearly independent by condition (i) in Definition  \ref{defi-admissible}. Now, $x_j\not\in U=\spn(x_i\mid i\in I)$ implies that $x_j$ is also linearly independent from these vectors. In other words, the vectors $x_i$ for $i\in I\cup \{j\}$ are linearly independent and the set $I\cup \{j\}$ satisfied condition (i) in Definition \ref{defi-admissible}. Thus, we can conclude that the set $I\cup \{j\}$ is admissible for $(x_1,\dots,x_k)$.

Let us now compare the weights of the admissible sets $I$ and $I\cup \{j\}$ with respect to $(x_1,\dots,x_k)$. Since $I=\mathcal{I}(x_1,\dots,x_k)$ must have the maximum weight among all admissible sets with respect to $(x_1,\dots,x_k)$, the weight of $I\cup \{j\}$ can be at most as large as the weight of $I$. However, the weight of $I$ is
\[(k+1)\cdot |I|+\left|\left\{\spn(\proj_{V/U}(x_j))\,\Big|\, j\in [k]\sm I\right\}\right|\leq (k+1)\cdot |I|+k<(k+1)\cdot |I\cup \{j\}|,\]
which means that the weight of $I\cup \{j\}$ is actually larger than the weight of $I$. This is a contradiction.
\end{proof}

\subsection{Proof of Proposition \ref{propo-weight}}

Let us now prove Proposition \ref{propo-weight} by induction on $w$.

For the base case $w=0$, we can take $C=1$ and $c=1$. Then for any subset $A\su \Fpn\sm\{0\}$ of size $\vert A\vert>C\cdot p^{(1-c)n}=1$, we can pick some vector $x\in A$ and consider the $k$-tuple $(x_1,\dots,x_k)=(x,\dots,x)\in A^k$. Note that by our assumption $a_{j,1}+\dots+a_{j,k}=0$ for $j=1,\dots,m$ made at the beginning of this section, this $k$-tuple $(x_1,\dots,x_k)$ is a solution to the system ($\star$). Furthermore, by the first statement in Claim \ref{claim-properties-weight} we have $\omega(x_1,\dots,x_k)\neq 0$. Hence  $\omega(x_1,\dots,x_k)> 0$ and we have proved the base case $w=0$.

Now let us assume that $w\geq 1$ and that Proposition \ref{propo-weight} holds for $w-1$ with constants $C'\geq 1$ and $0<c'\leq 1$. Also recall that we made the assumption $k\geq 2m+1+\lfloor w/(k+1)\rfloor$ in the statement of Proposition \ref{propo-weight}.

First, consider the case that $w\in \{0,(k+1),2\cdot (k+1),\dots,(k-1)\cdot (k+1)\}$. In this case, we can take $C=C'$ and $c=c'$. Indeed, for any subset $A\su \Fpn\sm\{0\}$ of size $|A|> C\cdot p^{(1-c)n}=C'\cdot p^{(1-c')n}$, by the induction hypothesis the system  ($\star$) has a solution $(x_1,\dots,x_k)\in A^k$ with weight $\omega(x_1,\dots,x_k)> w-1$. Since $\omega(x_1,\dots,x_k)\not\in \{0,(k+1),2\cdot (k+1),\dots,(k-1)\cdot (k+1)\}$ by Claim \ref{claim-properties-weight}, we have $\omega(x_1,\dots,x_k)\neq w$ and therefore $\omega(x_1,\dots,x_k)> w$. This shows the desired statement in Proposition \ref{propo-weight} for $w$.

We may therefore from now on assume that $w\not\in \{0,(k+1),2\cdot (k+1),\dots,(k-1)\cdot (k+1)\}$. By the assumption $k\geq 2m+1+\lfloor w/(k+1)\rfloor$, we have $w<k\cdot (k+1)$. Hence $w$ is not divisible by $k+1$.

Since $k\geq 2m+1+\lfloor w/(k+1)\rfloor$, we have $k-\lfloor w/(k+1)\rfloor\geq 2m+1$, so the constant $\Gamma_{p,m,k-\lfloor w/(k+1)\rfloor}<p$ considered in Section \ref{sect-background} is well-defined. Let us write $\Gamma=\Gamma_{p,m,k-\lfloor w/(k+1)\rfloor}$, then $1\leq \Gamma<p$.

Now, $p/\Gamma>1$ and $c'>0$, so there is some $c>0$ such that $p^c=(p/\Gamma)^{c'/(k-1)}$. Note that $c\leq 1$ since $c'\leq 1$ and $\Gamma\geq 1$. So $0<c\leq 1$, as desired.

We need to prove that there exists a constant $C\geq 1$ such that for any non-negative integer $n$ and any subset $A\su \Fpn\sm\{0\}$ of size $|A|> C\cdot p^{(1-c)n}$, the system  ($\star$) has a solution $(x_1,\dots,x_k)\in A^k$ with weight $\omega(x_1,\dots,x_k)> w$.

Note that by making the constant $C$ large enough, we may assume that $n$ is sufficiently large such that $(p/\Gamma)^n>(2kp^2)^{(2k+1)(k-1)}$.

Hence it suffices to prove that for any $n$ with $(p/\Gamma)^n>(2kp^2)^{(2k+1)(k-1)}$ and any subset $A\su \Fpn\sm\{0\}$ of size $|A|>4C'(2kp)^{2k+1}\cdot p^{(1-c)n}$, there is a solution $(x_1,\dots,x_k)\in A^k$ to  ($\star$) with $\omega(x_1,\dots,x_k)> w$.

So let us assume for contradiction that $(p/\Gamma)^n>(2kp^2)^{(2k+1)(k-1)}$ and $|A|>4C'(2kp)^{2k+1}\cdot p^{(1-c)n}$ and that every solution $(x_1,\dots,x_k)\in A^k$ the system ($\star$) of has weight $\omega(x_1,\dots,x_k)\leq w$.

Since $p^c=(p/\Gamma)^{c'/(k-1)}$ by the definition of $c$, we have
\begin{equation}\label{eq-size-A}
|A|>4C'(2kp)^{2k+1}\cdot p^{(1-c)n}=4C'(2kp)^{2k+1}\cdot p^n\cdot \left(\frac{\Gamma}{p}\right)^{c'n/(k-1)}
\end{equation}
Using that $0<c'\leq 1$ and $\Gamma<p$ and $C'\geq 1$, this furthermore implies
\begin{equation}\label{eq-size-A-vs-Gamma}
|A|>4C'(2kp)^{2k+1}\cdot p^n\cdot \left(\frac{\Gamma}{p}\right)^{c'n/(k-1)}\geq 4C'(2kp)^{2k+1}\cdot p^n\cdot \left(\frac{\Gamma}{p}\right)^{n}\geq (2kp)^{2k+1}\cdot \Gamma^n.
\end{equation}

Our proof strategy is somewhat similar to the proof of Theorem \ref{theo-distinct-induction} in Section \ref{sect-proof-distinct}, again using a probabilistic subspace sampling argument. We will again consider a random subspace $V\su \Fpn$ of a suitably chosen dimension. This time, our goal is to find a fairly large subset $A^*\su A\cap V$ such that there are no solutions $(x_1,\dots,x_k)\in (A^*)^k$ to the system ($\star$) of weight $\omega(x_1,\dots,x_k)=w$. By our assumption on $A$, this means that every solution  $(x_1,\dots,x_k)\in (A^*)^k$ to the system ($\star$) must satisfy $\omega(x_1,\dots,x_k)\leq w-1$. We will then obtain a contradiction by applying the induction hypothesis for $w-1$ to the set $A^*$.

In order to be able to obtain the desired set $A^*\su A\cap V$ not containing any solutions $(x_1,\dots,x_k)\in (A^*)^k$ to ($\star$) with $\omega(x_1,\dots,x_k)=w$, we will bound the total number of solutions $(x_1,\dots,x_k)\in A^k$ to ($\star$) in the set $A$ with $\omega(x_1,\dots,x_k)=w$. More precisely, since the relevant probabilities for the subspace sampling argument depend on the dimension $\dim\spn(x_1,\dots,x_k)$ for each solution $(x_1,\dots,x_k)\in A^k$, we will actually bound the number of solutions $(x_1,\dots,x_k)\in A^k$ to ($\star$) with $\omega(x_1,\dots,x_k)=w$ and $\dim\spn(x_1,\dots,x_k)=r$ for every possible value of this dimension $r$. The next claim analyzes these possible values.

\begin{claim}\label{claim-possible-dimensions}
For every solution $(x_1,\dots,x_k)\in A^k$ to the system ($\star$) of weight $\omega(x_1,\dots,x_k)=w$, we have $|\mathcal{I}(x_1,\dots,x_k)|=\lfloor w/(k+1)\rfloor$ and $\lfloor w/(k+1)\rfloor+1\leq \dim\spn(x_1,\dots,x_k)\leq k$.
\end{claim}
\begin{proof}
Fix a solution $(x_1,\dots,x_k)\in A^k$ to ($\star$) with $\omega(x_1,\dots,x_k)=w$. By  Claim \ref{claim-properties-weight} we have $|\mathcal{I}(x_1,\dots,x_k)|= \lfloor\omega(x_1,\dots,x_k)/(k+1)\rfloor=\lfloor w/(k+1)\rfloor$.

For the second part of the claim, the upper bound $\dim\spn(x_1,\dots,x_k)\leq k$ is clear. For the lower bound, note that by Claim \ref{claim-properties-weight} we have
\[\dim\spn(x_1,\dots,x_k)\geq \omega(x_1,\dots,x_k)/(k+1)=w/(k+1)>\lfloor w/(k+1)\rfloor,\]
where in the last step we used that $w$ is not divisible by $k+1$. Hence  we must have $\dim\spn(x_1,\dots,x_k)\geq \lfloor w/(k+1)\rfloor+1$, as desired.
\end{proof}

We remark that one can actually obtain a stronger upper bound for $\dim\spn(x_1,\dots,x_k)$ than the bound in the claim above by taking into account the linear relations imposed by the system ($\star$). However, the trivial upper bound $\dim\spn(x_1,\dots,x_k)\leq k$ suffices for our argument.

In light of Claim \ref{claim-possible-dimensions}, for every $r=\lfloor w/(k+1)\rfloor+1,\dots,k$, we need to bound the number of solutions $(x_1,\dots,x_k)\in A^k$ to ($\star$) with $\omega(x_1,\dots,x_k)=w$ and $\dim\spn(x_1,\dots,x_k)=r$. We will count these solutions by distinguishing all possibilities for the set $\mathcal{I}(x_1,\dots,x_k)$, noting that by Claim \ref{claim-possible-dimensions}, $\mathcal{I}(x_1,\dots,x_k)\su [k]$ is always a subset of size $\lfloor w/(k+1)\rfloor$.

The following lemma is the key step in order to obtain the desired bound. It gives a structural property for solutions $(x_1,\dots,x_k)\in A^k$ to ($\star$) with $\omega(x_1,\dots,x_k)=w$ where  $\mathcal{I}(x_1,\dots,x_k)=I$ for some fixed set $I\su [k]$ (of size $|I|=\lfloor w/(k+1)\rfloor$) and where the vectors $x_i\in A$ for $i\in I$ are fixed. In the proof of this lemma, we will use the slice rank polynomial method in the form of Corollary \ref{coro-slice-rank-partitioned tensor}. Recall that we defined $\Gamma=\Gamma_{p,m,k-\lfloor w/(k+1)\rfloor}<p$.

\begin{lemma}\label{lemma-few-disjoint-solutions}
Fix a subset $I\su [k]$ of size $|I|=\lfloor w/(k+1)\rfloor$, and fix vectors $x_i\in A$ for $i\in I$. Let $U=\spn(x_i\mid i\in I)$. Suppose that $(x_1^{(\l)},\dots,x_k^{(\l)})\in A^k$ for $\l=1,\dots,L$ is a list of solutions to the system ($\star$) such that for all $\l=1,\dots,L$ we have $\omega(x_1^{(\l)},\dots,x_k^{(\l)})=w$ and $\mathcal{I}(x_1^{(\l)},\dots,x_k^{(\l)})=I$ and $x_i^{(\l)}=x_i$ for all $i\in I$. Furthermore, suppose that the sets $\{\spn(\proj_{\Fpn/U}(x_j^{(\l)}))\mid j\in [k]\sm I\}$ are disjoint for all $\l=1,\dots,L$. Then we must have $L\leq k^{k+1}\cdot \Gamma^n$.
\end{lemma}
\begin{proof}
Suppose for contradiction that $L> k^{k+1}\cdot \Gamma^n$. For each $\l=1,\dots,L$ we can apply Lemma \ref{lemma-structure-partition-exists} to $(x_1^{(\l)},\dots,x_k^{(\l)})$ and obtain a partition $[k]\sm I=J_1^{(\l)}\cup \dots\cup J_{t_\l}^{(\l)}$ with $|J_h^{(\l)}|\geq 2$ for $h=1,\dots,t_\l$ and distinct one-dimensional subspaces $W_1^{(\l)},\dots,W_{t_\l}^{(\l)}\su \Fpn/U$ such that $\spn(\proj_{\Fpn/U}(x_j^{(\l)}))=W_h^{(\l)}$ for all $h=1,\dots,t_\l$  and all $j\in J_h^{(\l)}$. Note that by the condition $|J_h^{(\l)}|\geq 2$ for $h=1,\dots,t_\l$, the sets $J_h^{(\l)}$ are all non-empty and we have $t_\l\leq k$ for all $\l=1,\dots,L$.

Let us now distinguish all possibilities for the partitions $[k]\sm I=J_1^{(\l)}\cup \dots\cup J_{t_\l}^{(\l)}$. Since $t_\l\leq k$, there are at most $k^k$ possibilities for such a partition. Hence, as $L> k^{k+1}\cdot \Gamma^n$, for more than $k\cdot \Gamma^n$ different $\l$ we must obtain the same partition $[k]\sm I=J_1^{(\l)}\cup \dots\cup J_{t_\l}^{(\l)}$. In other words, there exists a partition $[k]\sm I=J_1\cup \dots\cup J_t$ which occurs for more than $k\cdot \Gamma^n$ different $\l$. By relabeling, we may assume that this partition occurs for $\l=1,\dots,L'$ for some $L'>k\cdot \Gamma^n$.

To summarize, we now have a fixed partition $[k]\sm I=J_1^{(\l)}\cup \dots\cup J_{t}^{(\l)}$ with $|J_h|\geq 2$ for $h=1,\dots,t$ and for each $\l=1,\dots, L'$ (where $L'>k\cdot \Gamma^n$) we have distinct one-dimensional subspaces $W_1^{(\l)},\dots,W_{t}^{(\l)}\su \Fpn/U$ such that $\spn(\proj_{\Fpn/U}(x_j^{(\l)}))=W_h^{(\l)}$ for all $h=1,\dots,t$  and all $j\in J_h$. Now, for each $\l=1,\dots,L'$ we have $\{\spn(\proj_{\Fpn/U}(x_j^{(\l)}))\mid j\in [k]\sm I\}=\{W_1^{(\l)},\dots,W_{t}^{(\l)}\}$ and by the assumptions in the lemma these sets are disjoint for all  $\l=1,\dots,L'$. This means that the subspaces $W_h^{(\l)}\su \Fpn/U$ are distinct for all $h=1,\dots,t$ and all $\l=1,\dots,L'$.

Recall that for every $\l=1,\dots,L'$ we have $\mathcal{I}(x_1^{(\l)},\dots,x_k^{(\l)})=I$. In particular, this means that the set $I$ is admissible for $(x_1^{(\l)},\dots,x_k^{(\l)})$. Hence, recalling that $x_i^{(\l)}=x_i$ for all $i\in I$, by condition (ii) in Definition \ref{defi-admissible} we have $x_j^{(\l)}\not\in \spn(x_i\mid i\in I)=U$ for all $j\in [k]\sm I$ and all $\l=1,\dots,L'$. Furthermore, using that $L'\geq 1$ as $L'>k\cdot \Gamma^n$, condition (i) in Definition \ref{defi-admissible} implies that the vectors $x_i$ for $i\in I$ are linearly independent.

For the moment, choose any $\l\in \{1,\dots,L'\}$ (using that $L'\geq 1$). Since we have $\mathcal{I}(x_1^{(\l)},\dots,x_k^{(\l)})=I$ and $\omega(x_1^{(\l)},\dots,x_k^{(\l)})=w$, the weight of the admissible set $I$ with respect to $(x_1^{(\l)},\dots,x_k^{(\l)})$ must be equal to $w$. On the other hand, this weight is
\[(k+1)\cdot |I|+\left|\left\{\spn(\proj_{\Fpn/U}(x_j^{(\l)}))\,\Big|\, j\in [k]\sm I\right\}\right|=(k+1)\cdot |I|+\left|\left\{W_1^{(\l)},\dots,W_{t}^{(\l)}\right\}\right|=(k+1)\cdot |I|+t.\]
Hence we can conclude that
\begin{equation}\label{eq-value-t-for-partition}
t=w-(k+1)\cdot |I|.
\end{equation}

For ease of notation, let us assume for the rest of the proof of this lemma that the set $I\su [k]$ consists of the last $\vert I\vert=\lfloor w/(k+1)\rfloor$ indices, i.e.\ $I=\{k-|I|+1,\dots,k\}$ (otherwise we can just relabel the indices). Given the fixed vectors  $x_i\in A$ for $i\in I=\{k-|I|+1,\dots,k\}$, we can now interpret ($\star$) as a (non-homogeneous) system of $m$ linear equations in the $k-\vert I\vert$ variables $x_1,\dots,x_{k-\vert I\vert}$. For each $\l=1,\dots,L'$ the $(k-\vert I\vert)$-tuple $(x_1^{(\l)},\dots,x_{k-\vert I\vert}^{(\l)})$ is a solution to this system of equations, since $(x_1^{(\l)},\dots,x_{k-\vert I\vert}^{(\l)},x_{k-|I|+1},\dots,x_k)=(x_1^{(\l)},\dots,x_k^{(\l)})$ is a solution to ($\star$) (recall the assumption that $x_i^{(\l)}=x_i$ for all $i\in I$).

Our goal is now to apply Corollary \ref{coro-slice-rank-partitioned tensor} to this system of equations and the solutions $(x_1^{(\l)},\dots,x_{k-\vert I\vert}^{(\l)})$ for $\l=1,\dots,L'$. The following claim states that the condition in the statement of Corollary \ref{coro-slice-rank-partitioned tensor} is satisfied.

\begin{claim}\label{claim-cross-solution-sat-condition}
We have $|\{ \l_j \mid j\in J_h\}|=1$ for $h=1,\dots,t$ for any choice of $\l_1,\dots,\l_{k-|I|}\in [L']$ such that $(x_1^{(\l_1)},\dots,x_{k-\vert I\vert}^{(\l_{k-|I|})})$ is a solution to ($\star$) when interpreted as a system in the first $k-\vert I\vert$ variables after fixing the given $x_i\in A$ for $i\in I=\{k-|I|+1,\dots,k\}$.
\end{claim}

Before proving this claim, let us first finish the rest of the proof of the lemma. Recall that $k-\vert I\vert=k-\lfloor w/(k+1)\rfloor\geq 2m+1$. We can therefore apply  Corollary \ref{coro-slice-rank-partitioned tensor} to the system of of $m$ linear equations in the $k-\vert I\vert$ variables $x_1,\dots,x_{k-\vert I\vert}$ obtained from ($\star$) after fixing the given $x_i\in A$ for $i\in I$. By Claim \ref{claim-cross-solution-sat-condition} the solutions $(x_1^{(\l)},\dots,x_{k-\vert I\vert}^{(\l)})$ for $\l=1,\dots,L'$ to this system and the partition $[k]\sm I=J_1\cup\dots\cup J_t$ satisfy the conditions in Corollary \ref{coro-slice-rank-partitioned tensor}. Hence the corollary implies that $L'\leq k\cdot \Gamma_{p,m,k-\lfloor w/(k+1)\rfloor}^n=k\cdot \Gamma^n$. This  contradicts the lower bound $L'>k\cdot \Gamma^n$ from above. This contradiction finishes the proof of the lemma, apart from proving Claim \ref{claim-cross-solution-sat-condition}.

\begin{proof}[Proof of Claim \ref{claim-cross-solution-sat-condition}]
Let us fix $\l_1,\dots,\l_{k-|I|}\in [L']$ such that $(x_1^{(\l_1)},\dots,x_{k-\vert I\vert}^{(\l_{k-|I|})})$ is a solution to this system, meaning that $(x_1^{(\l_1)},\dots,x_{k-\vert I\vert}^{(\l_{k-|I|})},x_{k-|I|+1},\dots,x_k)$ is a solution to the original system ($\star$). By our assumption on the set $A$, this solution $(x_1^{(\l_1)},\dots,x_{k-\vert I\vert}^{(\l_{k-|I|})},x_{k-|I|+1},\dots,x_k)\in A^k$ to ($\star$) must have weight
\begin{equation}\label{eq-weight-cross-solution}
\omega (x_1^{(\l_1)},\dots,x_{k-\vert I\vert}^{(\l_{k-|I|})},x_{k-|I|+1},\dots,x_k)\leq w.
\end{equation}
We claim that the set $I=\{k-|I|+1,\dots,k\}$ is admissible for $(x_1^{(\l_1)},\dots,x_{k-\vert I\vert}^{(\l_{k-|I|})},x_{k-|I|+1},\dots,x_k)$. Indeed, as shown above the vectors $x_i$ for $i\in I$ are linearly independent, so condition (i) in Definition \ref{defi-admissible} is satisfied. We also proved above that $x_j^{(\l)}\not\in \spn(x_i\mid i\in I)=U$ for all $j\in [k]\sm I$ and all $\l=1,\dots,L'$. In particular, $x_j^{(\l_j)}\not\in \spn(x_i\mid i\in I)$ for all $j\in [k]\sm I$ and so condition (ii) is satisfied as well. Hence the set $I=\{k-|I|+1,\dots,k\}$ is admissible for $(x_1^{(\l_1)},\dots,x_{k-\vert I\vert}^{(\l_{k-|I|})},x_{k-|I|+1},\dots,x_k)$.

Now, by (\ref{eq-weight-cross-solution}) the weight of the admissible set $I=\{k-|I|+1,\dots,k\}$ with respect to the solution $(x_1^{(\l_1)},\dots,x_{k-\vert I\vert}^{(\l_{k-|I|})},x_{k-|I|+1},\dots,x_k)$ to ($\star$)  is at most $w$. Hence
\[w\geq (k+1)\cdot |I|+\left|\left\{\spn(\proj_{\Fpn/U}(x_j^{(\l_j)}))\,\Big|\, j\in [k]\sm I\right\}\right|=(k+1)\cdot |I|+\left|\bigcup_{h=1}^{t}\left\{W_h^{(\l_j)}\,\big|\, j\in J_h\right\}\right|\]
Here, for the second step we used that $[k]\sm I=J_1\cup \dots\cup J_t$ is a partition and that $\spn(\proj_{\Fpn/U}(x_j^{(\l)}))=W_h^{(\l)}$ for all $h=1,\dots,t$, all $j\in J_h$ and all $\l=1,\dots,L'$. Together with (\ref{eq-value-t-for-partition}) this yields
\[t=w-(k+1)\cdot |I|\geq \left|\bigcup_{h=1}^{t}\left\{W_h^{(\l_j)}\,\big|\, j\in J_h\right\}\right|=\sum_{h=1}^{t}|\{\l_j\mid j\in J_h\}|,\]
where in the last step we used that the spaces $W_h^{(\l)}\su \Fpn/U$ are distinct for all $h=1,\dots,t$ and all $\l=1,\dots,L'$. Since the sets $J_h$ for $h=1,\dots,t$ are non-empty (as $|J_h|\geq 2$), we have $|\{\l_j\mid j\in J_h\}|\geq 1$ for $h=1,\dots,t$. Hence the previous inequality implies that we must have $|\{\l_j\mid j\in J_h\}|= 1$ for all $h=1,\dots,t$, as desired.
\end{proof}

This finishes the proof of Lemma \ref{lemma-few-disjoint-solutions}.
\end{proof}

Lemma \ref{lemma-few-disjoint-solutions} states that for a fixed subset $I\su [k]$ (of size $|I|=\lfloor w/(k+1)\rfloor$) and fixed vectors $x_i\in A$ for $i\in I$, there cannot be too many solutions $(x_1,\dots,x_k)\in A^k$ to  ($\star$) with $\omega(x_1,\dots,x_k)=w$ and $\mathcal{I}(x_1,\dots,x_k)=I$ such that the sets $\{\spn(\proj_{\Fpn/U}(x_j))\mid j\in [k]\sm I\}$ are disjoint for all of these solutions. We will now use this lemma to bound the total number of solutions $(x_1,\dots,x_k)\in A^k$ to  ($\star$) with $\omega(x_1,\dots,x_k)=w$ and $\mathcal{I}(x_1,\dots,x_k)=I$ for fixed $I$ and fixed $x_i\in A$ for $i\in I$ (more precisely, we bound the number of such solution with $\dim\spn(x_1,\dots,x_k)=r$ for each fixed $r$), as stated in the following lemma. Roughly speaking, the proof strategy for this lemma is to choose a maximal collection of solutions $(x_1',\dots,x_k')$ of the desired form for which the sets $\{\spn(\proj_{\Fpn/U}(x_j'))\mid j\in [k]\sm I\}$ are disjoint, and then to use that for every solution $(x_1,\dots,x_k)$ satisfying the desired conditions we must have $\proj_{\Fpn/U}(x_t)\in \{\spn(\proj_{\Fpn/U}(x_j'))\mid j\in [k]\sm I\}$ for some $t\in [k]\sm I$ and some $(x_1',\dots,x_k')$ in the chosen collection. This means that there are only relatively few possibilities for $x_t\in A$, and we will be able to derive the desired bound on the number of possible $(x_1,\dots,x_k)$.

\begin{lemma}\label{lemma-number-solutions-I-fixed}
Fix a subset $I\su [k]$ of size $|I|=\lfloor w/(k+1)\rfloor$, and fix vectors $x_i\in A$ for $i\in I$. Furthermore, fix an integer $r$ with $\lfloor w/(k+1)\rfloor+1\leq r\leq k$. Then the number of solutions $(x_1,\dots,x_k)\in A^k$ to  ($\star$) with $\omega(x_1,\dots,x_k)=w$ , $\mathcal{I}(x_1,\dots,x_k)=I$ and $\dim\spn(x_1,\dots,x_k)=r$ is at most $k^{k+3}2^kp^{rk}\cdot \Gamma^n\cdot |A|^{r-|I|-1}$.
\end{lemma}
\begin{proof}
Let $U=\spn(x_i\mid i\in I)$. Furthermore, let us fix a list of solutions $(x_1^{(\l)},\dots,x_k^{(\l)})\in A^k$ for $\l=1,\dots,L$ to ($\star$) of maximum possible length $L$ such that for all $\l=1,\dots,L$ we have $\omega(x_1^{(\l)},\dots,x_k^{(\l)}))=w$ and $\mathcal{I}(x_1^{(\l)},\dots,x_k^{(\l)}))=I$ and $x_i^{(\l)}=x_i$ for all $i\in I$, and such that the sets $\{\spn(\proj_{\Fpn/U}(x_j^{(\l)}))\mid j\in [k]\sm I\}$ are disjoint for all $\l=1,\dots,L$. 

Then by Lemma \ref{lemma-few-disjoint-solutions} we must have $L\leq k^{k+1}\cdot \Gamma^n$ (and in particular, such a list of maximum possible length exists). For every $\l=1,\dots,L$ and every $j\in [k]\sm I$, define
\[W_j^{(\l)}=\spn(x_i^{(\l)}\mid i\in I\cup\{j\})=\spn(x_i\mid i\in I)+\spn(x_j^{(\l)})=U+\spn(x_j^{(\l)}),\]
and note that each $W_j^{(\l)}$ is a subspace of $\Fpn$ of dimension at most $\vert I\vert +1=\lfloor w/(k+1)\rfloor+1\leq r$. Hence the union $\bigcup_{\l=1}^{L}\bigcup_{j\in [k]\sm I} W_j^{(\l)}$ is a subset of $\Fpn$ of size
\begin{equation}\label{eq-choices-x-t}
\left|\bigcup_{\l=1}^{L}\bigcup_{j\in [k]\sm I} W_j^{(\l)}\right|\leq L\cdot k\cdot p^r\leq k^{k+1}\cdot \Gamma^n\cdot k\cdot p^r=k^{k+2}p^r\cdot \Gamma^n.
\end{equation}

\begin{claim}\label{claim-index-t-set-S}
Suppose $x_j\in A$ for $j\in [k]\sm I$ are vectors such that $(x_1,\dots,x_k)\in A^k$ is a solution to  ($\star$) with $\omega(x_1,\dots,x_k)=w$ , $\mathcal{I}(x_1,\dots,x_k)=I$ and $\dim\spn(x_1,\dots,x_k)=r$. Then there exists an index $t\in [k]\sm I$ and a subset $S\su [k]\sm (I\cup \{t\})$of size $|S|=r-|I|-1$ such that $x_t\in \bigcup_{\l=1}^{L}\bigcup_{j\in [k]\sm I} W_j^{(\l)}$ and such that $x_1,\dots,x_k\in \spn(x_i\mid i\in I\cup \{t\}\cup S)$.
\end{claim}
\begin{proof}
By maximality of our chosen list of solutions $(x_1^{(\l)},\dots,x_k^{(\l)})$, it cannot be possible to extend this list by taking $(x_1^{(L+1)},\dots,x_k^{(L+1)})=(x_1,\dots,x_k)$. Since $(x_1,\dots,x_k)$ satisfies all of the other conditions, this means that we must have
\[\{\spn(\proj_{\Fpn/U}(x_j))\mid j\in [k]\sm I\}\cap \{\spn(\proj_{\Fpn/U}(x_j^{(\l)}))\mid j\in [k]\sm I\}\neq \emptyset\]
or some $\l\in [L]$. Hence we can find $t\in [k]\sm I$ and $j\in [k]\sm I$ and $\l\in [L]$ such that $\spn(\proj_{\Fpn/U}(x_t))=\spn(\proj_{\Fpn/U}(x_j^{(\l)}))$. In particular, we have $\proj_{\Fpn/U}(x_t)\in \spn(\proj_{\Fpn/U}(x_j^{(\l)}))$. This means that $x_t\in U+\spn(x_j^{(\l)})=W_j^{(\l)}$.

Thus, we have found the desired index $t\in [k]\sm I$ with $x_t\in \bigcup_{\l=1}^{L}\bigcup_{j\in [k]\sm I} W_j^{(\l)}$. It remains to find a set $S\su [k]\sm (I\cup \{t\})$ with the desired conditions.

Since $\mathcal{I}(x_1,\dots,x_k)=I$, the set $I$ is admissible for $(x_1,\dots,x_k)$, and so by condition (i) in Definition \ref{defi-admissible} the vectors $x_i$ for $i\in I$ are linearly independent. Furthermore, by condition (ii) in Definition \ref{defi-admissible} we have $x_t\not\in \spn(x_i\mid i\in I)=U$. Hence the vectors $x_i$ for $i\in I\cup \{t\}$ are linearly independent. We can therefore find a subset  $S\su [k]\sm (I\cup \{t\})$ such that the vectors $x_i$ for $i\in I\cup \{t\}\cup S$ form a basis of the space $\spn(x_1,\dots,x_k)$. Then we clearly have $x_1,\dots,x_k\in \spn(x_i\mid i\in I\cup \{t\}\cup S)$. Furthermore, as  $\dim\spn(x_1,\dots,x_k)=r$, we also have $\vert S\vert=r-|I|-1$, as desired.
\end{proof}

In order to prove the lemma, we need to show that the number of choices for $(x_j\mid j\in [k]\sm I)$ satisfying the conditions in Claim \ref{claim-index-t-set-S} is at most $k^{k+3}2^kp^{rk}\cdot \Gamma^n\cdot |A|^{r-|I|-1}$. We will count by distinguishing all possibilities for the index $t\in [k]\sm S$ and the set $S\su [k]\sm (I\cup \{t\})$ obtained from  Claim \ref{claim-index-t-set-S}. There are clearly at most $k$ possibilities for $t$ and at most $2^k$ possibilities for $S$.

Hence it suffices to prove that for every fixed $t\in [k]\sm S$ and every fixed set $S\su [k]\sm (I\cup \{t\})$ of size $|S|=r-|I|-1$ there are at most $k^{k+2}p^{rk}\cdot \Gamma^n\cdot |A|^{r-|I|-1}$ possibilities for choosing vectors $x_j\in A$ for $j\in [k]\sm I$ such that $x_t\in \bigcup_{\l=1}^{L}\bigcup_{j\in [k]\sm I} W_j^{(\l)}$ and $x_1,\dots,x_k\in \spn(x_i\mid i\in I\cup \{t\}\cup S)$.

By (\ref{eq-choices-x-t}) the number of possibilities for $x_t$ is at most $k^{k+2}p^r\cdot \Gamma^n$. For every $j\in S$, the number of possibilities for $x_j\in A$ is at most $|A|$. Finally, after making all these choices, for each of the remaining $j\in [k]\sm(I\cup \{t\}\cup S)$, we have at most $p^r$ choices for $x_j$, since $x_j\in \spn(x_i\mid i\in I\cup \{t\}\cup S)$ and $\dim \spn(x_i\mid i\in I\cup \{t\}\cup S)\leq |I|+1+|S|=r$ (recall that the vectors $x_i$ for $i\in I$ are fixed). Thus, the number of possibilities for choosing $x_j\in A$ for $j\in [k]\sm I$ with the above properties is indeed at most
\[k^{k+2}p^r\cdot \Gamma^n\cdot |A|^{|S|}\cdot (p^r)^{k-|I|-|S|-1}=k^{k+2}p^r\cdot \Gamma^n\cdot |A|^{r-|I|-1}\cdot (p^r)^{k-r}\leq k^{k+2}p^{rk}\cdot \Gamma^n\cdot |A|^{r-|I|-1}.\]
This finishes the proof of the lemma.
\end{proof}

By adding up the bound in Lemma \ref{lemma-number-solutions-I-fixed} over all possible choices of the subset $I\su [k]$ and the vectors $x_i\in A$ for $i\in I$, we can show the following corollary.

\begin{corollary}\label{coro-number-solutions-total}
For any fixed $r$ with $\lfloor w/(k+1)\rfloor+1\leq r\leq k$, the number of solutions $(x_1,\dots,x_k)\in A^k$ to the system ($\star$) with $\omega(x_1,\dots,x_k)=w$ and $\dim\spn(x_1,\dots,x_k)=r$ is at most $(2k)^{2k}p^{rk}\cdot \Gamma^n\cdot |A|^{r-1}$.
\end{corollary}
\begin{proof}
By Claim \ref{claim-possible-dimensions} every such solution $(x_1,\dots,x_k)$ satisfies $|\mathcal{I}(x_1,\dots,x_k)|=\lfloor w/(k+1)\rfloor$.

We claim that for each set $I\su [k]$ with $|I|=\lfloor w/(k+1)\rfloor$, there are at most $k^{k+3}2^{k}p^{rk}\cdot \Gamma^n\cdot |A|^{r-1}$ solutions $(x_1,\dots,x_k)\in A^k$ to the system ($\star$) with $\omega(x_1,\dots,x_k)=w$ and $\dim\spn(x_1,\dots,x_k)=r$ and $\mathcal{I}(x_1,\dots,x_k)=I$. Indeed, there are $|A|^{|I|}$ possibilities to choose vectors $x_i\in A$ for $i\in I$, and after fixing these vectors the desired number of solutions $(x_1,\dots,x_k)\in A^k$ is by Lemma \ref{lemma-number-solutions-I-fixed} at most $k^{k+3}2^kp^{rk}\cdot \Gamma^n\cdot |A|^{r-|I|-1}$. So in total we have indeed at most $|A|^{|I|}\cdot k^{k+3}2^kp^{rk}\cdot \Gamma^n\cdot |A|^{r-|I|-1}=k^{k+3}2^{k}p^{rk}\cdot \Gamma^n\cdot |A|^{r-1}$ solutions for any given $I$.

Since the number of possible subsets $I\su [k]$ with $|I|=\lfloor w/(k+1)\rfloor$ is clearly at most $2^k$, the number of solutions $(x_1,\dots,x_k)\in A^k$ as in the statement of the corollary is at most
\[2^k\cdot k^{k+3}2^{k}p^{rk}\cdot \Gamma^n\cdot |A|^{r-1}\leq (2k)^{2k}p^{rk}\cdot \Gamma^n\cdot |A|^{r-1}.\]
Here, we used that $k\geq 3$ since $k\geq 2m+1+\lfloor w/(k+1)\rfloor$ and $m\geq 1$.
\end{proof}

As mentioned above, we now perform a random subspace sampling argument. Choose the unique integer $d$ such that 
\begin{equation}\label{eq-definition-d}
\frac{1}{(2k)^{2k+1}\cdot p^{2k+1}}\cdot\left(\frac{p}{\Gamma}\right)^{n/(k-1)}<p^d \leq \frac{1}{(2k)^{2k+1}\cdot p^{2k}}\cdot \left(\frac{p}{\Gamma}\right)^{n/(k-1)}.
\end{equation}

\begin{claim}\label{claim-value-d}
We have $2\leq d\leq n$.
\end{claim}
\begin{proof}
Recall that we assumed that $(p/\Gamma)^n>(2kp^2)^{(2k+1)(k-1)}$. Hence
\[p^d>\frac{1}{(2k)^{2k+1}\cdot p^{2k+1}}\cdot (2kp^2)^{2k+1}=p^{2k+1},\]
and consequently $d\geq 2k+1\geq 2$. For the upper bound on $d$, recall that $1\leq \Gamma<p$, so $p^d\leq (p/\Gamma)^{n/(k-1)}\leq p^{n}$ and therefore $d\leq n$.
\end{proof}

Let us now consider a uniformly random $d$-dimensional subspace $V\su \Fpn$. The following claims give useful bounds for the expected number of vectors in $A\cap V$ and the expected number of solutions $(x_1,\dots,x_k)\in A^k$ with $\omega(x_1,\dots,x_k)=w$ and $x_1,\dots,x_k\in V$.

\begin{claim}\label{claim-expected-vectors-sampling2}
We have
\[\EE[|A\cap V|]> \frac{3}{4}\cdot \frac{p^d}{p^n}\cdot  |A|.\]
\end{claim}
\begin{proof}
Recall that $0\not\in A$. So by Lemma \ref{lemma-subspace-probability} (applied with $s=1$), for each vector $x\in A$ we have
\[\PP[x\in V]=\frac{p^d-1}{p^n-1}> \frac{3}{4}\cdot \frac{p^d}{p^n}.\]
Here we used that $p^d-1\geq (3/4)p^d$ since $p\geq 2$ and $d\geq 2$ (see Claim \ref{claim-value-d}). Now, adding this up over all $x\in A$ gives the desired bound.
\end{proof}

\begin{claim}\label{claim-expected-solutions-dimension}
For any $r$ with $\lfloor w/(k+1)\rfloor+1\leq r\leq k$, the expected number of solutions $(x_1,\dots,x_k)\in A^k$ to ($\star$) with $\omega(x_1,\dots,x_k)=w$ and $\dim\spn(x_1,\dots,x_k)=r$ such that $x_1,\dots,x_k\in V$ is at most
\[\frac{1}{2k}\cdot \frac{p^d}{p^n}\cdot  |A|\]
\end{claim}
\begin{proof}
By Corollary \ref{coro-number-solutions-total}, the total number of solutions $(x_1,\dots,x_k)\in A^k$ to ($\star$) with $\omega(x_1,\dots,x_k)=w$ and $\dim\spn(x_1,\dots,x_k)=r$ is at most $(2k)^{2k}p^{rk}\cdot \Gamma^n\cdot |A|^{r-1}$. For each of these solutions, by Lemma \ref{lemma-subspace-probability} (applied to a list of $r$ linearly independent vectors among $x_1,\dots,x_k$), the probability of having $x_1,\dots,x_k\in V$ is at most $(p^d/p^n)^r$.

Thus, in the case $r\geq 2$ the expected number of solutions in the claim statement is by $|A|\leq p^n$ and (\ref{eq-definition-d}) at most
\begin{multline*}
(2k)^{2k}p^{rk}\cdot \Gamma^n\cdot |A|^{r-1}\cdot \left(\frac{p^d}{p^n}\right)^r\leq \frac{1}{2k}\cdot (2k)^{2k+1}p^{rk}\cdot \Gamma^n\cdot |A|\cdot (p^n)^{r-2}\cdot  \left(\frac{p^d}{p^n}\right)^{r-1}\cdot \frac{p^d}{p^n}\\
=\frac{1}{2k}\cdot \frac{p^d}{p^n}\cdot  |A|\cdot(2k)^{2k+1}p^{rk}\cdot \frac{\Gamma^n}{p^n}\cdot (p^d)^{r-1}\\
\leq \frac{1}{2k}\cdot \frac{p^d}{p^n}\cdot  |A|\cdot(2k)^{2k+1}p^{rk}\cdot \left(\frac{1}{(2k)^{2k+1}\cdot p^{2k}}\right)^{r-1}\cdot \left(\frac{\Gamma}{p}\right)^{n\cdot \left(1-\frac{r-1}{k-1}\right)}\leq \frac{1}{2k}\cdot \frac{p^d}{p^n}\cdot  |A|
\end{multline*}
where in the last step we used that $\Gamma<p$. It remains to consider the case $r=1$. In this case the expected number of solutions as in the claim statement is at most
\[(2k)^{2k}p^{rk}\cdot \Gamma^n\cdot |A|^{r-1}\cdot \left(\frac{p^d}{p^n}\right)^r=(2k)^{2k}p^{k}\cdot \Gamma^n\cdot \frac{p^d}{p^n}\leq \frac{1}{2k}\cdot  |A|\cdot \frac{p^d}{p^n},\]
where we used that $|A|\geq (2kp)^{2k+1}\cdot \Gamma^n\geq (2k)^{2k+1}p^{k}\cdot \Gamma^n$ by (\ref{eq-size-A-vs-Gamma}).
\end{proof}

\begin{corollary}\label{coro-expected-solutions-total}
The expected number of solutions $(x_1,\dots,x_k)\in A^k$ to ($\star$) with $\omega(x_1,\dots,x_k)=w$ and $x_1,\dots,x_k\in V$ is at most
\[\frac{1}{2}\cdot \frac{p^d}{p^n}\cdot  |A|.\]
\end{corollary}
\begin{proof}
By Claim \ref{claim-possible-dimensions}, for every  solution $(x_1,\dots,x_k)\in A^k$ to ($\star$) with $\omega(x_1,\dots,x_k)=w$ we must have $\lfloor w/(k+1)\rfloor+1\leq \dim\spn(x_1,\dots,x_k)\leq k$. Hence, the claim follows from Claim \ref{claim-expected-solutions-dimension} by summing over all integers $r$ with $\lfloor w/(k+1)\rfloor+1\leq r\leq k$ (noting that there are at most $k$ possibilities for $r$).
\end{proof}

Let $Z$ be the number of solutions $(x_1,\dots,x_k)\in A^k$ to ($\star$) with $\omega(x_1,\dots,x_k)=w$ and $x_1,\dots,x_k\in V$. By Corollary \ref{coro-expected-solutions-total}, we have $\EE[Z]\leq (1/2)\cdot (p^d/p^n)\cdot |A|$. Together with Claim \ref{claim-expected-vectors-sampling2} this yields
\[\EE[|A\cap V|-Z]> \frac{3}{4}\cdot \frac{p^d}{p^n}\cdot  |A|-\frac{1}{2}\cdot \frac{p^d}{p^n}\cdot  |A|=\frac{1}{4}\cdot \frac{p^d}{p^n}\cdot  |A| .\]
So let us fix an outcome for the $d$-dimensional subspace  $V\su \Fpn$ such that $|A\cap V|-Z> (1/4)\cdot (p^d/p^n)\cdot |A|$.

We can now define a subset $A^*\su A\cap V$ by deleting one vector from each solution $(x_1,\dots,x_k)\in A^k$ to ($\star$) with $\omega(x_1,\dots,x_k)=w$ and $x_1,\dots,x_k\in V$.  Then
\[|A^*|\geq |A\cap V|-Z> \frac{1}{4}\cdot \frac{p^d}{p^n}\cdot  |A|\]
and there do not exist any solutions $(x_1,\dots,x_k)\in (A^*)^k$ to ($\star$) with $\omega(x_1,\dots,x_k)=w$.

Our goal is now to obtain a contradiction by applying the induction hypothesis for $w-1$ to $A^*\su V\cong \Fp^d$. Note that using (\ref{eq-size-A}) and (\ref{eq-definition-d}), we have
\[|A^*|>\frac{1}{4}\cdot \frac{p^d}{p^n}\cdot  |A|\geq \frac{1}{4}\cdot \frac{p^d}{p^n}\cdot 4C'(2kp)^{2k+1} \cdot p^n\cdot \left(\frac{\Gamma}{p}\right)^{c'n/(k-1)}\geq C'\cdot p^d\cdot (2kp)^{c'(2k+1)}\cdot \left(\frac{\Gamma}{p}\right)^{c'n/(k-1)}\geq C'\cdot p^{(1-c')d}.\]
Since $V\cong \Fp^d$, we can interpret $A^*\su V$ as a subset of $\Fp^d$. As $|A^*|>C'\cdot p^{(1-c')d}$, by the induction hypothesis for $w-1$, there must be a solution $(x_1,\dots,x_k)\in (A^*)^k$ to the system  ($\star$) with $\omega(x_1,\dots,x_k)>w-1$. By our construction of $A^*$, we have $\omega(x_1,\dots,x_k)\neq w$. This means that we must have $\omega(x_1,\dots,x_k)>w$, but since $(x_1,\dots,x_k)\in (A^*)^k\su A^k$ this is a contradiction to our assumption on the set $A$. This  finishes the inductive proof of Proposition \ref{propo-weight}.

\section{Concluding Remarks}
\label{sect-concluding-remarks}

Recall that in our main result, Theorem \ref{theo-distinct}, we assumed that every $m\times m$ minor of the $m\times k$ matrix $(a_{j,i})_{j,i}$ is non-singular. As discussed in the introduction, it is not possible to remove this assumption completely, but it may be possible to weaken it in some ways.

In order for the conclusion of Theorem \ref{theo-distinct} to hold, it is certainly necessary to assume that no equation of the form $x_i-x_{i'}=0$ for distinct $i,i'\in [k]$ is in the span of the equations in the system ($\star$), since otherwise there do not exist any solution $(x_1,\dots,x_k)\in (\Fpn)^k$ to ($\star$) where $x_1,\dots,x_k$ are distinct. In other words (using our other assumption that $a_{j,1}+\dots+a_{j,k}=0$ for $j=1,\dots,m$, which is also necessary as discussed in the introduction), in Theorem \ref{theo-distinct} we certainly need to assume that the row-span of the $m\times k$ matrix $(a_{j,i})_{j,i}$ does not contain any vector with exactly two non-zero entries. 

If we assume that $a_{j,1}+\dots+a_{j,k}=0$ for $j=1,\dots,m$ and that the row-span of the $m\times k$ matrix $(a_{j,i})_{j,i}$ does not contain any vector with exactly two non-zero entries, then every subset $A\su \Fpn$ not containing solution $(x_1,\dots,x_k)\in A^k$ to ($\star$) with distinct $x_1,\dots,x_k$ must have size $|A|=o(p^n)$ as $n\to \infty$ (where $p$, $m$ and $k$ are fixed). This follows from an arithmetic removal lemma for solutions to systems of linear equations due to Kr\'{a}l', Serra, and Vena  \cite[Theorem 1]{kral-serra-vena} and independently Shapira \cite[Theorem 2.2]{shapira}. It would be plausible that one also has a bound of the form $|A|\leq C_{p,m,k}\cdot (\Gamma_{p,m,k}^*)^n$ with $\Gamma_{p,m,k}^*<p$ under these weaker assumptions, i.e.\ it would be plausible that Theorem \ref{theo-distinct} also holds with these weaker assumptions.

However, proving Theorem \ref{theo-distinct} under these weaker assumptions is most likely extremely difficult. In fact, proving such a statement (even with an assumption that the number of variables is very large in terms of the number of equations) would imply a bound of the form $|A|\leq C_{p,k}\cdot (\Gamma_{p,k}^*)^n$ with $\Gamma_{p,k}^*<p$ for the size of a $k$-term-progression-free subset $A\su \Fpn$. Indeed, for any (large) $K$, one can take ($\star$) to be the system of $k-1$ equations in $k+K$ variables consisting of the equations $x_i-2x_{i+1}+x_{i+2}=0$ for $i=1,\dots,k-2$ as well as the equation $Kx_{k}-x_{k+1}-\dots-x_{k+K}=0$. Then $a_{j,1}+\dots+a_{j,k}=0$ for $j=1,\dots,k-1$ and the row-span of the $(k-1)\times (k+K)$ matrix $(a_{j,i})_{j,i}$ does not contain any vector with exactly two non-zero entries. But for any solution $(x_1,\dots,x_{k+K})\in A^{k+K}$ with distinct $x_1,\dots,x_{k+K}$, the vectors $x_1,\dots,x_k$ form a non-constant $k$-term arithmetic progression. Hence any $k$-term-progression-free subset $A\su \Fpn$ in particular does not contain a solution $(x_1,\dots,x_{k+K})\in A^{k+K}$ to this system of equations with distinct $x_1,\dots,x_k$.

For $k\geq 4$, proving that a $k$-term-progression-free subset $A\su \Fpn$ has size $|A|\leq C_{p,k}\cdot (\Gamma_{p,k}^*)^n$ with $\Gamma_{p,k}^*<p$ is a big open problem, that has received the attention of many researchers, especially after Ellenberg and Gijswijt \cite{ellenberg-gijswijt} proved such a statement for $k=3$. Weakening the assumptions on the matrix $(a_{j,i})_{j,i}$ in Theorem \ref{theo-distinct} in the way discussed above is at least as difficult a problem, and therefore seems to be out of reach of current methods.

A more tractable problem might be to improve upon the value of the constant $\Gamma_{p,m,k}^*<p$ in Theorem \ref{theo-distinct} that our proof obtains. Given the inductive nature of the proof with the repeated subspace sampling arguments, this value is likely not optimal. It would be extremely interesting to determine whether one can take $\Gamma_{p,m,k}^*$ to be equal to the constant $\Gamma_{p,m,k}$ in Theorem \ref{theo-tao} as defined in (\ref{eq-defi-Gamma-tao}). Even in the case where ($\star$) consists only of one equation (i.e.\ in the case $m=1$) this is a widely open problem, and in the special case of the equation $x_1+\dots+x_p=0$ it has applications to bounding Erd\H{o}s-Ginzburg-Ziv constants (see \cite{fox-sauermann, naslund, sauermann}).


\begin{thebibliography}{99}

\bibitem{bateman-katz}  M. Bateman and N. H. Katz, \textit{New bounds on cap sets}, J. Amer. Math. Soc. \textbf{25} (2012), 585--613.

\bibitem{blasiak-et-al} J. Blasiak, T. Church, H. Cohn, J. A. Grochow, E. Naslund, W. F. Sawin, and C. Umans, \textit{On cap sets and the group-theoretic approach to matrix multiplication}, Discrete Anal. 2017, Paper No. 3, 27pp.

\bibitem{croot-lev-pach} E. Croot, V. F. Lev, and P. P. Pach, \textit{Progression-free sets in $\Z_4^n$ are exponentially small}, Ann. of Math. \textbf{185} (2017), 331--337.

\bibitem{edel} Y. Edel, \textit{Extensions of generalized product caps}, Des. Codes Cryptogr. \textbf{31} (2004), 5--14.

\bibitem{ellenberg-gijswijt} J. S. Ellenberg and D. Gijswijt, \textit{On large subsets of $\mathbb{F}_q^n$ with no three-term arithmetic progression}, Ann. of Math. \textbf{185} (2017), 339--343.

\bibitem{elsholtz-pach} C. Elsholtz and P. P. Pach, \textit{Caps and progression-free sets in $\mathbb{Z}_m^n$}. Des. Codes Cryptogr. \textbf{88} (2020), 2133--2170.

\bibitem{egz} P. Erd\H{o}s, A. Ginzburg, and A. Ziv, \textit{Theorem in the additive number theory}, Bull. Res. Council Israel \textbf{10F} (1961), 41--43.

\bibitem{erdos-turan-1936} P. Erd\H{o}s and P. Tur\'{a}n, \textit{On Some Sequences of Integers}, J. Lond. Math. Soc. \textbf{11} (1936), 261--264. 

\bibitem{fox-sauermann} J. Fox and L. Sauermann, \textit{Erd\H{o}s-Ginzburg-Ziv constants by avoiding three-term arithmetic progressions}, Electron. J. Combin. \textbf{25} (2018), Paper 2.14, 9 pp.

\bibitem{green-tao-1} B. Green and T. Tao, \textit{An inverse theorem for the Gowers $U^3(G)$-norm, with applications},  Proc. Edinb. Math. Soc. \textbf{51} (2008), 73--153.

\bibitem{green-tao-2} B. Green, and T. Tao, \textit{New bounds for Szemer\'{e}di's theorem, I: Progressions of length 4 in finite field geometries}, Proc. Lond. Math. Soc. \textbf{98} (2009), 365--392.

\bibitem{green-tao-2a} B. Green, and T. Tao, \textit{New bounds for Szemer\'{e}di's theorem, Ia: Progressions of length 4 in finite field geometries revisited}, preprint, 2012, arXiv:1205.1330.

\bibitem{grochow} J. A. Grochow, \textit{New applications of the polynomial method: the cap set conjecture and beyond}, Bull. Amer. Math. Soc. \textbf{56} (2019), 29--64. 

\bibitem{komlos-sulyok-szemeredi-1975} J. Koml\'{o}s, M. Sulyok, E. Szemer\'{e}di, \textit{Linear problems in combinatorial number theory}, Acta Math. Acad. Sci. Hungar. \textbf{26} (1975), 113--121. 

\bibitem{kral-serra-vena}  D. Kr\'{a}l', O. Serra, and L. Vena, \textit{A removal lemma for systems of linear equations over finite fields},  Israel J. Math. \textbf{187} (2012), 193--207.

\bibitem{lin-wolf}  Y. Lin, Y and J. Wolf, \textit{On subsets of $\mathbb{F}_q^n$ containing no $k$-term progressions}. European J. Combin. \textbf{31} (2010), 1398--1403.

\bibitem{meshulam} R. Meshulam, \textit{On subsets of finite abelian groups with no 3-term arithmetic progressions}, J. Combin. Theory Ser. A \textbf{71} (1995), 168--172.

\bibitem{mimura-tokushige-jcta} M. Mimura and N. Tokushige, \textit{Avoiding a shape, and the slice rank method for a system of equations}, preprint, 2019, arXiv:1909.10509.

\bibitem{mimura-tokushige-1} M. Mimura and N. Tokushige, \textit{Solving linear equations in a vector space over a finite fields I}, preprint, 2020, \url{http://www.cc.u-ryukyu.ac.jp/~hide/sol.pdf}.

\bibitem{mimura-tokushige-2} M. Mimura and N. Tokushige, \textit{Solving linear equations in a vector space over a finite fields II}, preprint, 2020, \url{http://www.cc.u-ryukyu.ac.jp/~hide/sol2.pdf}.

\bibitem{naslund} E. Naslund, \textit{Exponential Bounds for the Erd\H{o}s-Ginzburg-Ziv Constant},  J. Combin. Theory Ser. A \textbf{174} (2020), 105185, 19 pp.

\bibitem{ruzsa-1}  I. Z. Ruzsa, \textit{Solving a linear equation in a set of integers  I}, Acta Arith. \textbf{65} (1993), 259--282.

\bibitem{ruzsa-2}  I. Z. Ruzsa, \textit{Solving a linear equation in a set of integers II}, Acta Arith. \textbf{72} (1995), 385--397.

\bibitem{sauermann} L. Sauermann, \textit{On the size of subsets of $\mathbb{F}_p^n$ without $p$ distinct elements summing to zero}, preprint, 2019, arXiv:1904.09560.

\bibitem{sawin-tao} W. Sawin and T. Tao, \textit{Notes on the ``slice rank'' of tensors}, blog post, 2016, \url{https://terrytao.wordpress.com/2016/08/24/notes-on-the-slice-rank-of-tensors/}.

\bibitem{shapira} A. Shapira, \textit{A proof of Green's conjecture regarding the removal properties of sets of linear equations}, J. Lond. Math. Soc. \textbf{81} (2010), 355--373. 

\bibitem{tao} T. Tao, \textit{A symmetric formulation of the Croot-Lev-Pach-Ellenberg-Gijswijt capset bound}, blog post, 2016, \url{http://terrytao.wordpress.com/2016/05/18/a}.
\end{thebibliography}
\end{document}